\documentclass[12pt]{article}
\usepackage[utf8]{inputenc}
\usepackage{tikz}
\usepackage{amsthm}
\usepackage{amsmath}
\usepackage{amssymb}
\usepackage{amsthm}
\usepackage{amscd}
\usepackage[all]{xy}
\usepackage[margin=1in]{geometry}
\usepackage{comment}
\usepackage{pdfpages}
\usepackage[T1]{fontenc}
\usepackage{amsmath,amsfonts,amssymb,enumerate}
\usepackage{amscd}
\usepackage{graphicx}
\usepackage[T1]{fontenc}
\usepackage{amsmath,amsfonts,amssymb,enumerate}
\usepackage{amsthm} 
\usepackage[english]{babel}

\newtheorem{thm}{Theorem}[section]
\newtheorem{lem}[thm]{Lemma}
\newtheorem{prop}[thm]{Proposition}

\theoremstyle{definition}

\newtheorem{definition}[thm]{Definition}
\newtheorem{remark}[thm]{Remark}
\newtheorem{example}[thm]{Example}

\numberwithin{equation}{section}

\newcommand{\FF}{\mathbb{F}}

\begin{document}

\title{ Mapping resolutions of length three I. }

\author{Lorenzo Guerrieri \thanks{ Jagiellonian University, Instytut Matematyki, 30-348 Krak\'{o}w \textbf{Email address:} guelor@guelan.com
}  
\and Jerzy Weyman
\thanks{Jagiellonian University, Instytut Matematyki, 30-348 Krak\'{o}w \textbf{Email address:} jerzy.weyman@uj.edu.pl }} 
\maketitle
\begin{abstract}
\noindent We produce some interesting families of resolutions of length three by describing certain open subsets of the spectrum of the generic ring for such resolutions constructed in \cite{W18}.
\noindent MSC: 13D02, 13C05 \\
\noindent Keywords: free resolutions, generic ring
\end{abstract}

\section{Introduction}

In this paper we construct some interesting finite free resolutions of length three of formats $(1,n,n,1)$ (for $n$ even) and $(1,4,n,n-3)$. 
The construction is based on the construction of the generic ring ${\hat R}_{gen}$ given in \cite{W18}. 

Given a free resolution $\mathbb{F}_{\bullet}$ of a module over a commutative Noetherian ring $R$ (we will generally work over a regular local ring), we call \it format \rm the sequence of the Betti numbers of $\mathbb{F}_{\bullet}$.

In \cite{W18} are constructed a generic ring ${\hat R}_{gen}$ and a generic complex $\FF^{gen}_\bullet$ associated to each format of free resolution of length three. Such ring and complex have a universal property, in the sense that any free resolution of that format in a given commutative ring can be obtained by the generic one by homomorphic image. This construction requires powerful tools from representation theory. A Kac-Moody Lie algebra is canonically associated to each format, and it turns out that the generic ring is Noetherian if and only if this Lie algebra is finite dimensional. This last fact identifies a special family of formats (called Dynkin formats) for which this condition holds. Ideals and modules whose minimal free resolutions are of Dynkin format are similarly called Dynkin. Dynkin ideals include Gorenstein ideals, almost complete intersections and few other formats.

It turns out that for Dynkin formats, the fact that ${\hat R}_{gen}$ is Noetherian allows to study what kinds of explicit resolutions one can get from it by localization.
More specifically, for Dynkin formats the ring ${\hat R}_{gen}$ exhibits certain lowest weight-highest weight symmetry and this means that one can construct over it another acyclic complex $\FF^{top}_\bullet$, which is a mirror image of the generic complex $\FF^{gen}_\bullet$. Calculation of the matrices of the complex $\FF^{top}_\bullet$ involves calculation of the higher structure theorems which give tensors generating ${\hat R}_{gen}$.

The procedure we follow in this paper comes from that construction. We start with a well-known acyclic complex of free modules of length 3 and we calculate higher structure theorems for generic choices of factorizations that appear, using so-called defect variables. We end up with three matrices which by construction give the differentials of a certain acyclic complex over a polynomial ring in the variables needed to define a Hilbert-Burch complex and the defect variables. 
In this way we describe the points in certain open subsets in the spectrum of ${\hat R}_{gen}$. 
More precisely, we deal with certain slice of the spectrum of the generic ring constructed in \cite{W18} (see remark \ref{rmk:slice}) and thus the constructed complexes are automatically acyclic over a field of characteristic zero.

Our starting complex in this paper is either a split exact complex or a direct sum of a generic Hilbert-Burch complex and an identity map.
Our formats are the ones corresponding to the root system $D_n$ in the combinatorial scheme of \cite{W18}.
Even though the procedure that we use comes from the generic ring, we try to make the paper independent of \cite{W18}.

For the format $(1,n,n,1)$ with $n$ even our method involves calculating the multiplicative structure and the so-called Bruns cycles (see \cite{W18}) for known non-minimal resolutions of that format. We do it for the split exact complex and for a direct sum of a Hilbert-Burch complex of format $(1,n,n-1)$ and an identity map $R\rightarrow R$.
The three matrices we get (they form the top components of the so-called critical representations $W(d_3)$, $W(d_2)$ and $W(d_1)$, see section \ref{sec:basic}) give the differentials of an acyclic complex.

For the format $(1,4,n,n-3)$ we follow similar procedure based on the results of \cite{W18}. Again we calculate higher structure theorems for a split complex of this format and for direct sums of the Hilbert-Burch complex of format $(1,4,3)$ and the identity map $R^{n-3}\rightarrow R^{n-3}$.

The complexes $\FF^{top}_\bullet$ constructed in this way have been already defined in \cite{LW19}, and we expect it to be useful in investigating more subtle properties of open subsets of the spectrum of the generic ring ${\hat R}_{gen}$. 
In a forthcoming paper we intend to produce similar resolutions starting from different non-minimal complexes involving Koszul complex of format $(1,3,3,1)$.

The complexes we  construct resolve certain modules $R/I$ where the  ideals $I$ have  grade $2$ and $pd_R(R/I)=3$, where $R$ is a polynomial ring in the defect variables for a given format.
We also show with examples that the specializations of our resolutions occur  "in nature". One more application is a new proof (independent of linkage) that for the formats $(1,n,n,1)$ ($n$ even) and $(1,4,n,n-3)$, two relevant open subsets of the spectrum of generic ring ${\hat R}_{gen}$, called $U_{CM}$ and $U_{split}$ are the same. 
Here $U_{CM}$ denotes the set of all prime ideals $P$ of ${\hat R}_{gen}$ for which the localization of the dual complex $(\FF^{gen}_\bullet )^*$ at $P$ is acyclic, i.e. the points where the localization of the generic module $H_0 (\FF^{gen}_\bullet)$ is perfect.
The open set $U_{split}$ consists of points for which the complex $\FF^{top}_\bullet$ is split exact.
A proof based on linkage of the fact that this two sets coincides for the above formats is discussed in the paper \cite{CVW18}.

The paper is organized as follows. 
In section \ref{sec:results} we state the results of the paper, i.e. we present two acyclic complexes of length 3.
In section \ref{sec:basic} we recall the basic facts related to generic ring, critical representations and defect variables.

In section \ref{sec:1nn1} we deal with format $(1,n,n,1)$ and calculate complexes $\FF^{top}_\bullet$ for our two classes of complexes. We give proofs independent of \cite{W18} of the fact that the complexes we get are acyclic.
In section \ref{sec:14nn-3}  we do the same for the format $(1,4,n,n-3)$. 

For each format, we also construct the complex $\FF^{top}_\bullet$ obtained starting from a split exact complex, by calculating higher structure theorems in similar way.
 Such complex can be also computed by localization of the appropriate complex $\FF^{top}_\bullet$ computed starting from a direct sum of a generic Hilbert-Burch and an identity map.
 
In Section \ref{sec:usplit}, we prove that the open sets $U_{CM}$ and $U_{split}$ are the same for formats $(1,n,n,1)$ and $n$ even and $(1,4,n,n-3)$. 
Finally in section \ref{sec:examples} we give examples of interesting modules whose resolutions are specializations of the complexes constructed in section \ref{sec:results}.

\section{Two complexes $\FF^{top}_\bullet$.}\label{sec:results}

We present the main results of the paper. Given an acyclic complex $\FF_{\bullet}$ of Dynkin format (except $A_n$, $D_n$ with $n$ odd, and the module format $(2,5,5,2)$), it is always possible to canonically construct a second complex of the same format called $\FF^{top}_\bullet$. The differentials of $\FF^{top}_\bullet$ are obtained computing the highest nonzero components of the three critical representation of the Lie algebra associated to the format of $\FF$ (see Section \ref{sec:basic} for the needed definitions). 

In this section we state the two main theorems describing the complex $\FF^{top}_\bullet$ associated to a complex $\FF_{\bullet}$ equal to the direct sum of an Hilbert-Burch complex and an identity map. The explicit method to compute the differentials will be explained in the next sections of this paper.

\subsection{The complex $\FF^{top}_\bullet$ for the format $(1,n,n,1)$}

Let $n \geq 4$ be an even number and let $X=  \left( X_{ij} \right)$ be a generic $n \times n-1$ matrix on variables $\lbrace X_{ij} \rbrace$. 
Let $k$ be an arbitrary field  and set
$R_0=k[X_{ij} ]$. Consider the free resolution 
\begin{equation}
\label{hilb1nn1}
\FF: 0 \longrightarrow F_3 \buildrel{d_3}\over\longrightarrow  F_2 \buildrel{d_2}\over\longrightarrow F_1 \buildrel{d_1}\over\longrightarrow R_0 
\end{equation}
on the free $R_0$-modules $F_1$,$F_2$,$F_3$ having bases $\lbrace e_1, \ldots, e_n \rbrace$, $\lbrace f_1, \ldots, f_{n} \rbrace$, $\lbrace g_1 \rbrace$ and differentials
$$d_3 =\bmatrix 0  \\ \vdots  \\ 0 \\ 1 \endbmatrix , 
d_2=\bmatrix X_{11} & X_{12} &  \ldots & X_{1,n-1} & 0\\  X_{21} & X_{22} & \ldots &  X_{2,n-1} & 0\\ \vdots & \vdots & \ldots & \vdots & \vdots \\ X_{n1} & X_{n2} &  \ldots & X_{n,n-1} & 0 \endbmatrix , 
d_1=\bmatrix Y_1 & Y_2 & \ldots &Y_n \endbmatrix $$ where $(-1)^{i+1}Y_i$ is equal to the maximal minor of the matrix $X$ obtained deleting the $i$-th row. 

Let us denote by $ X_{\hat{k}}^{\hat{i},\hat{j}} $ the minor of $X$ obtained by deleting the rows $i,j$ and the column $k$. Let $b_{ij}$ be new indeterminates over the ring $R_0$ defined for any $1 \leq i < j \leq n$ 
(such variables will be introduced in detail in Definition \ref{defect}).

  \begin{definition}
 \label{pfaffians}
 Let $B$ be the $n\times n$ skew symmetric matrix  in the variables $b_{ij}$, an define by $P_{i_1, \ldots, i_t}$ the pfaffian of the submatrix of $B$ defined by the $i_1, \ldots, i_t$ rows and columns, and similarly by $P_{\hat{i_1}, \ldots, \hat{i_t}}$ the pfaffian of the submatrix obtained removing the $i_1, \ldots, i_t$ rows and columns. We simply call $P$ the pfaffian of the whole matrix.
 \end{definition}
 
 \begin{definition}\label{def:rfor1nn1} For the remainder of this section we work over the polynomial ring $R$ over $R_0$
 on the defect variables $b_{ij}$.
 \end{definition}

\begin{thm}
\label{thhbn}
The complex $\FF^{top}_\bullet$ associated to the complex \ref{hilb1nn1} is equal to the minimal free resolution of the ideal $J=(u_1,\ldots,u_n) \subseteq R$ generated by the entries of $v^{(1)}_{top}$. In particular 
$u_i := \sum_{j \neq i}(-1)^{i+j}Y_j P_{\hat{i},\hat{j}}$.
The other differentials of $\FF^{top}_\bullet$  are:
 $$ v^{(2)}_{top}= \bmatrix 
-\sum_{j=1}^n X_{j1}b_{j1} & \ldots & -\sum_{j=1}^n X_{j,n-1}b_{j1} &  Y_1 \\  
\vdots & \ldots & \vdots & \vdots \\
\vdots & \ldots & \vdots & \vdots  \\ 
-\sum_{j=1}^n X_{j1}b_{jn}  & \ldots & -\sum_{j=1}^n X_{j,n-1}b_{jn}  &  Y_n  \\
 \endbmatrix, 
 v^{(3)}_{top}= \bmatrix v_1 \\ \vdots \\ v_{n-1} \\ P
 \endbmatrix,  $$
 where  $ v_k = \sum_{1 \leq i< j \leq n} (-1)^{i+j} X_{\hat{k}}^{\hat{i},\hat{j}} P_{\hat{i},\hat{j}} $ Moreover, $\FF^{top}_\bullet$ is acyclic.
\end{thm}


\subsection{The complex $\FF^{top}_\bullet$ for the format $(1,4,m+3,m)$}

Let $m \geq 1$ and let $X=  \left( X_{ij} \right)$ be a generic $4 \times 3$ matrix on variables $\lbrace X_{ij} \rbrace$. Let $k$ be a field of characteristic zero and set
$R_0=k[X_{ij} ]$. Consider the free resolution 
\begin{equation}
\label{hilb14m+3m}
\FF: 0 \longrightarrow F_3 \buildrel{d_3}\over\longrightarrow  F_2 \buildrel{d_2}\over\longrightarrow F_1 \buildrel{d_1}\over\longrightarrow R_0 
\end{equation}
on the free $R_0$-modules $F_1$,$F_2$,$F_3$ having bases $\lbrace e_1, e_2, e_3, e_4\rbrace$, $\lbrace f_1, \ldots, f_{m+3} \rbrace$, $\lbrace g_1, \ldots, g_m\rbrace$ and differentials
$$d_3 =\bmatrix 0 & \ldots & 0 \\ 0 & \ldots  & 0 \\ 0 & \ldots & 0 \\ & I_n & \endbmatrix , d_2=\bmatrix X_{11} & X_{12} & X_{13} & 0 & \ldots  & 0\\  X_{21} & X_{22} & X_{23} &0 & \ldots  & 0\\ X_{31} & X_{32} & X_{33} &0 & \ldots  & 0\\ X_{41} & X_{42} & X_{43} &0 & \ldots  & 0 \endbmatrix , d_1=\bmatrix Y_1 & Y_2 & Y_3 &Y_4 \endbmatrix $$ where $(-1)^{i+1}Y_i$ is equal to the maximal minor of the matrix $X$ obtained deleting the $i$-th row. 

Denote by $ X_{\hat{k}}^{i,j} $ the $2 \times 2$ minor of $X$ obtained by taking only the rows $i,j$ and deleting the column $k$. 



Let $b_{ij}^k,c_{ut}$ be new indeterminates over the ring $R_0$ defined for any $1 \leq i < j \leq 4$, $1 \leq k \leq m$. and $1 \leq u < t \leq m.$ We give the following definitions:

\begin{definition}
\label{minors-pfaffians}
We denote by $B$ the matrix of size $m \times 6$ whose entries are the variables $b_{ij}^k$ and by $B_{i_1j_1, \ldots, i_rj_r}^{k_1\cdots k_r}$ the minor obtained considering the square submatrix of size $r \leq m$ on all the variables $b_{ij}^k$ such that $\lbrace i,j \rbrace \in \lbrace \lbrace i_1,j_1 \rbrace, \ldots, \lbrace i_r,j_r \rbrace \rbrace$, 
$k \in \lbrace k_1, \ldots, k_r \rbrace$. We also consider some mixed-pfaffians of $B$ with the notation $$ P_{u,t}= \frac{1}{2}[b_{12}^ub_{34}^t - b_{13}^ub_{24}^t+ b_{14}^ub_{23}^t+b_{34}^ub_{12}^t-b_{24}^ub_{14}^t+b_{23}^ub_{14}^t  ], $$ for $1 \leq u,t \leq m$.
Similarly, considering the skew-symmetric $m \times m$ matrix on the variables $c_{ut}$, we denote by $\Gamma_{i_1, \ldots, i_t}$ and $\Gamma_{\hat{i_1}, \ldots, \hat{i_t}}$ the pfaffians obtained by respectively taking or removing the $i_1, \ldots, i_t$ rows and columns. If $m$ is even, we denote by $\Gamma$ the pfaffian of the whole matrix.
\end{definition}

 \begin{definition}\label{def:rfor1nn1} For the remainder of this section we work over the polynomial ring $R$ over $R_0$
 on the defect variables $b_{ij}^k, c_{ut}$.
 \end{definition}



\begin{thm}
\label{thhbm}
Let $m \geq 4$. The complex $\FF^{top}_\bullet$ is the minimal free resolution over $R$ of an ideal $J_m^{\prime}=(w_1,w_2,w_3,w_4)$ generated by the entries of $v^{(1)}_{top}$. Assume $i,j,k,r=1,2,3,4$. If $m=2s$ is even
$$ w_i= Y_i\Gamma +\sum_{u,t}(Y_jB^{ut}_{ik,ir}-Y_kB^{ut}_{ij,ir}+Y_rB^{ut}_{ij,ik})\Gamma_{\hat{u},\hat{t}},$$ 
If $m=2s+1$ is odd
$$w_i= Y_i \sum_{u,t,v}B^{utv}_{jk,jr,kr} \Gamma_{\hat{u},\hat{t},\hat{v}} +\sum_{u}(Y_jb^{u}_{kr}-Y_kb^{u}_{jr}+Y_rb^{u}_{jk})\Gamma_{\hat{u}}.$$ In both ideals the sums are taken over all the required $1 \leq u < t < v \leq m.$
The other matrices of the top maps associated to the complex \ref{hilb14m+3m} are:
$$ v^{(3)}_{top}= \bmatrix 
 \sum (-1)^{i+j+1} X_{\hat{1}}^{ij}b_{ij}^1 & \sum (-1)^{i+j+1} X_{\hat{1}}^{ij}b_{ij}^2  & \ldots  & \sum (-1)^{i+j+1} X_{\hat{1}}^{ij}b_{ij}^m
 \\  \sum (-1)^{i+j} X_{\hat{2}}^{ij}b_{ij}^1 & \sum (-1)^{i+j} X_{\hat{2}}^{ij}b_{ij}^2  & \ldots  & \sum (-1)^{i+j} X_{\hat{2}}^{ij}b_{ij}^m
 \\ \sum (-1)^{i+j+1} X_{\hat{3}}^{ij}b_{ij}^1 & \sum (-1)^{i+j+1} X_{\hat{3}}^{ij}b_{ij}^2  & \ldots  & \sum (-1)^{i+j+1} X_{\hat{3}}^{ij}b_{ij}^m
 \\ P_{1,1} & P_{1,2}+c_{12}  & \ldots  & P_{1,m}+c_{1m} 
 \\ \vdots & \vdots  & \vdots  & \vdots
 \\  P_{m,1}+c_{m1}    &  P_{m,2}+c_{m2}   &   \ldots  & P_{m,m}
 \endbmatrix,  $$
  where all the sums are taken over $1 \leq i < j \leq 4$,
  $$
   v^{(2)}_{top}= \bmatrix 
\lbrace e_1,f_1 \rbrace & \lbrace e_1,f_2 \rbrace & \ldots & \lbrace e_1,f_{m+3} \rbrace   \\ 
 \lbrace e_2,f_1 \rbrace & \lbrace e_2,f_2 \rbrace & \ldots & \lbrace e_2,f_{m+3} \rbrace  \\ 
\lbrace e_3,f_1 \rbrace & \lbrace e_3,f_2  \rbrace & \ldots & \lbrace e_3,f_{m+3} \rbrace  \\ 
\lbrace e_4,f_1 \rbrace & \lbrace e_4,f_2 \rbrace  & \ldots & \lbrace e_4,f_{m+3} \rbrace  \\ 
 \endbmatrix $$
  where, for $i,j,k,r=1,2,3$, $l \leq 3$, $h\geq 4$ and $1 \leq u < t < v \leq m$, if $m$ is even:
 $$\lbrace e_i,f_l \rbrace=  X_{il}\Gamma +\sum_{u,t}(X_{jl}B^{ut}_{jr,jk}-X_{kl}B^{ut}_{kj,kr}+X_{rl}B^{ut}_{rk,rj})\Gamma_{\hat{u},\hat{t}} ,  $$ 
$$\lbrace e_i,f_h \rbrace= Y_{i}\sum_{u,t,v \neq h-3}B^{utv}_{kr,jr,jk}\Gamma_{\hat{u},\hat{t},\hat{v},\widehat{h-3}} +\sum_{u \neq h-3}(Y_{j}b^{u}_{kr}-Y_{k}b^{u}_{jr}+Y_{r}b^{u}_{jk})\Gamma_{\hat{u},\widehat{h-3}}.  $$
 If instead $m$ is odd:  
 $$\lbrace e_i,f_l \rbrace=  X_{il} \sum_{u,t,v}B^{utv}_{ij,ik,ir} \Gamma_{\hat{u},\hat{t},\hat{v}} +\sum_{u}(X_{jl}b^{u}_{ij}-X_{kl}b^{u}_{ik}+X_{rl}b^{u}_{ir})\Gamma_{\hat{u}},  $$ 
$$\lbrace e_i,f_h \rbrace= Y_i \Gamma_{\widehat{h-3}} +\sum_{u,t \neq h-3}(Y_jB^{ut}_{ir,ik}-Y_kB^{ut}_{ir,ij}+Y_rB^{ut}_{ik,ij})\Gamma_{\hat{u}, \hat{t}, \widehat{h-3}}.  $$
Moreover, $\FF^{top}_\bullet$ is acyclic.
\end{thm}

\begin{remark} We conjecture that the complex $\FF^{top}_\bullet$ of Theorem \ref{thhbm} is acyclic when $k$ is a field of characteristic $\ne 2$.
\end{remark}


\section{General setting}\label{sec:basic}

Here we describe the general setting for the next sections of this article and the procedure that we are using.
Let $R_0$ be a commutative Noetherian ring of characteristic zero (usually it will be a polynomial ring in several variables) and consider the acyclic complex of $R_0$-modules
\begin{equation}
\label{complex}
\FF: 0 \longrightarrow F_3 \buildrel{d_3}\over\longrightarrow  F_2 \buildrel{d_2}\over\longrightarrow F_1 \buildrel{d_1}\over\longrightarrow F_0 = R_0. 
\end{equation}

We denote by $r_i$ the rank of $F_i$ and say that the complex $\FF$ has format $(1,r_1,r_2,r_3)$. 
The basis of $F_1,F_2,F_3$ will be respectively denoted by $\lbrace e_1, \ldots, e_{r_1}\rbrace$, $\lbrace f_1, \ldots, f_{r_2} \rbrace$, $\lbrace g_1, \ldots, g_{r_3}\rbrace$. In our case we will consider different examples of complexes of formats $(1,n,n,1)$ and $(1,4,m+3,m)$ for $n \geq 4$, $m \geq 1$. They correspond (according to the combinatorial scheme of \cite{W18}) to the Dynkin diagram $D_n$.

To construct a generic ring for free resolutions of format $D_n$, we define the defect variables in the following way:

\begin{definition}
\label{defect}
For $1 \leq i,j \leq r_1$ and $1 \leq k,u,t \leq r_3$, we introduce two sets of indeterminates $b_{ij}^{k}$ and $c_{ut}$ over the ring $R_0$. These are called \it defect variables. \rm      
We set 
$$ R=R_0[\lbrace b_{ij}^k \rbrace, \lbrace c_{ut} \rbrace] $$
 modulo by all the relations necessary to make the defect variables skew-symmetric in $i,j$ and $u,t$, that is $b_{ij}^{k}= - b_{ji}^{k}$, $b_{ii}^{k}=0$, $c_{ut}= - c_{tu}$, $c_{uu}=0$. Notice that for the format $(1,n,n,1)$ no variables of the form $c_{ut}$ are introduced.
\end{definition}

\begin{remark}\label{rmk:slice} The ring $R$ is related to the generic ring ${\hat R}_{gen}$ introduced in \cite{W18}. Let $r_3$ denote the rank of the third differential $d_3$ of the complex $\FF^{gen}_\bullet$. Let also $U(r_3, d_3)$ be the open set in $Spec ({\hat R}_{gen})$ of points such that the ideal of maximal minors of $d_3$ is nonzero. 
The ring $R$ is the coordinate ring of a slice $V$ of $U(r_3, d_3)$
corresponding to points where the particular $r_3\times r_3$ minor of the differential $d_3$ is a unit, when the complex $\FF^{gen}_\bullet$ is reduced to canonical form.
By a slice we mean a closed subset $V$ such that every $GL(\FF^{gen})_\bullet$-orbit in $U(r_3, d_3)$ intersects $V$.
Thus for all our purposes we can treat $R$ as a localization of ${\hat R}_{gen}$.
Thus it is clear that both complexes $\FF^{gen}_\bullet$ and $\FF^{top}_\bullet$ constructed in \cite{LW19} are acyclic over $R$.
The complex $\FF$ we start with is just the restriction of the complex $\FF^{gen}_\bullet$ to our set.
The complexes we construct are restrictions of $\FF^{top}_\bullet$ to our set.
Note also that the defect variables correspond to the positive part of Lie algebra $\mathfrak{g}(T_{p,q,r})$ introduced in \cite{W18}. Furthermore, recall that because of the highest-lowest weight symmetry in ${\hat R}_{gen}$ (with respect to the Lie algebra $\mathfrak{g}(T_{p,q,r})$) the roles of the complexes $\FF^{gen}_\bullet$ and $\FF^{top}_\bullet$ are symmetric.
\end{remark}
Throughout the paper we will work with so-called critical representations $W(d_3)$, $W(d_2)$, $W(d_1)$. These are the isotypic components of the generic ring ${\hat R}_{gen}$ with respect to the action of the group $GL(F_2)\times GL(F_0)$, corresponding to three differentials in the generic complex $\FF^{gen}_\bullet$.

For the format $(1,n,n,1)$ we deal with the Lie algebra
${\mathfrak{g}}(D_n)={\mathfrak{so}} (F^*_1\oplus F_1 )$.
The critical representations are:
$$W(d_3)=F^*_2\otimes [\oplus_{k=0}^{{n\over 2}}S_{1-k}F_3\otimes\bigwedge^{2k}F_1],$$
$$W(d_2)=F_2\otimes [F_1^*\oplus F_3^*\otimes F_1],$$
$$W(d_1)={\mathbb{C}} \otimes [\sum_{k=0}^{{n\over 2}-1} S_{-k}F_3\otimes\bigwedge^{2k+1}F_1].$$
For the format $(1,4,m+3,m)$ the Lie algebra corresponding to $T_{p,q,r}$ is ${\mathfrak{g}}(D_n)={\mathfrak{so}} (F^*_3\oplus\bigwedge^2 F_1\oplus F_3)$ and critical representations are:
$$W(d_3)=F_2^*\otimes [F_3\oplus \bigwedge^2 F_1\oplus F_3^*\otimes\bigwedge^4 F_1],$$
$$W(d_2)=F_2\otimes [F_1^*\otimes\bigwedge^{even}F_3^*\oplus F_1\otimes\bigwedge^{odd}F_3^*],$$
$$W(d_1)=\mathbb{C}\otimes [F_1^*\otimes\bigwedge^{odd}F_3^*\oplus F_1\otimes\bigwedge^{even}F_3^*].$$
One can easily see that the maps corresponding to the zero-graded components of these representations are exactly the differentials $d_3,d_2,d_1$ of the complex $\FF$. The components of degree one correspond to maps $v^{(3)}_1: \bigwedge^2 F_1 \to F_2$, $v^{(2)}_1:  F_1 \otimes F_2 \to F_3$, $v^{(1)}_1: \bigwedge^3 F_1 \to F_3$ and represent the well-known multiplicative structure on free resolutions of length 3 (see \cite{BE77}).

Our aim is to describe a general method to calculate the maps corresponding to all the graded components of $W(d_3),W(d_2),W(d_1)$ in terms of elements of $R$ and of the defect variables.
This can be done iteratively for large values of $n,m$ and, even if we do it for particular examples of interesting complexes, this same method can be applied to any resolution of the same format.
Since we are dealing with free resolutions of Dynkin type only finitely many maps are nonzero and the maps $v^{(3)}_{top}$,$v^{(2)}_{top}$,$v^{(1)}_{top}$ associated to the largest nonzero graded components are well-defined.
 For format $(1,n,n,1)$ with $n$ even and for any case of format $(1,4,m+3,m)$ these top maps are the differentials of a complex $\FF^{top}_\bullet$ defined over the ring $R$ and having the same format of $\FF$. In these cases, this complex is acyclic and hence its specializations to other commutative rings could be used to define classes of ideals having free resolution of a given format.


We give some extra indication on the notation that we are going to use. For an element $\alpha \in F_i$, we simply write $d(\alpha)$ to denote $d_i(\alpha) \in F_{i-1}$. Similarly, if $\alpha$ is an element of some exterior power, tensor product or a Schur functor on the modules $F_1,F_2,F_3$, $d(\alpha)$ is defined using the standard Leibniz rule.
For example, to compute the product $e_i^.e_j \in F_2$ (corresponding to the image of $e_i \wedge e_j$ by the map $v^{(3)}_1: \bigwedge^2 F_1 \to F_2$), we set $$d(e_i^.e_j)= d(e_i)e_j - d(e_j)e_i$$ and, using exactness of the complex $\FF$, we can find an element $\beta \in F_2$ such that $d(\beta)= d(e_i^.e_j)$. It now follows that $e_i^.e_j$ can be expressed by the sum of $\beta$ with any element of the kernel of $f_2$. Defect variables are used to define a generic element of this kernel, using the fact that this is equal to the image of $d_3$. Indeed, in this case we write $$e_i^.e_j = \beta + b_{ij}^1 d(g_1)+ \ldots + b_{ij}^m d(g_{r_3}).$$
For compactness of notation we prefer to set:

\begin{equation}
\label{bvar}
G_{ij}:=b_{ij}^1 g_1+ \ldots + b_{ij}^m g_{r_3}, 
 \end{equation}
and then $d(G_{ij})=b_{ij}^1 d(g_1)+ \ldots + b_{ij}^m d(g_{r_3}).$

Similar computation methods will be used for all the maps $v^{(i)}_j$. 
For the defect variables $c_{ut}$, we use the following compact notation: 

\begin{equation}
\label{cvar}
 C= \sum_{1\leq u < t \leq m} c_{ut} (g_u \wedge g_t) 
 \end{equation}
 and, similarly, $d(C)= \sum_{1\leq u < t \leq m} c_{ut} d(g_u \wedge g_t). $
 Moreover, higher forms of $C$ corresponding to generic pfaffians in the variables $ c_{ut} $ are introduced: 
 
 \begin{definition}
 \label{highC}
 Setting $ C^{(1)} := C $, inductively for $s \geq 2$ define
 $C^{(s)} \in \bigwedge^{2s} F_3$ 
  such that $$ d( C^{(s)}) := d(C) \wedge C^{(s-1)}. $$ This is well defined since $d(C^{(s-1)} \wedge d(C)) = 0$.
  By convention, set
 $ C^{(0)} := 1 $, $C^{(-1)} :=  0$.
 \end{definition}

The entries of the matrix of differentials of the complex $\FF^{top}_\bullet$ that we compute explicitly will be expressed as a combination of elements of $R_0$, minors or pfaffians of the generic matrix in the variables $ b_{ij}^k $ and pfaffians of the generic matrix in the variables $c_{ut}$.

In general we will use the notation $e_i^.e_j$, $e_i^.f_h$, $ e_i^.e_j^.e_k $ to express the result of the multiplication maps $v^{(3)}_1$,$v^{(2)}_1$,$v^{(1)}_1$. For higher maps $v^{(l)}_s$ we will write $v^{(l)}_s(e_{i_1}, \ldots, e_{i_t})$ to say that we apply the map on vectors $ e_{i_1}, \ldots, e_{i_t} $ or $v^{(l)}_s(\hat{e_{i_1}}, \ldots, \hat{e_{i_t}})$ to say that we are applying the map to all the generators of $F_1$ except $ e_{i_1}, \ldots, e_{i_t} $. The analogous notation is used for vector $f_i$,$g_i$ when needed.

\section{Format $(1,n,n,1)$}\label{sec:1nn1}

In this section we consider two complexes of formats $(1,n,n,1)$ 
and we describe how to compute all the maps $v^{(i)}_j$ and the complex $\FF^{top}_\bullet$. We provide an explicit computation for the split exact case, since, even if it can be obtained by localizing the complex obtained in the Hilbert-Burch case, it is instructive to describe our computation from the beginning in this case, which is the easiest. The result is presented in Theorem \ref{thsplitn} and all the required computations will follow it.
Along the same lines of this computation, one can compute the complex $\FF^{top}_\bullet$ associated to any given acyclic complex of length 3 of this format. 

In the second part of this section we consider the Hilbert-Burch case and describe the computation of the differentials of the complex presented in Theorem \ref{thhbn}.

\subsection{Split exact complex}
Let us start by choosing $R_0$ to be any field of characteristic zero.
Let $n = 2s \geq 4$ and consider the acyclic complex 
\begin{equation}
\label{splitcomplexn}
\FF : 0 \longrightarrow F_3 \buildrel{d_3}\over\longrightarrow  F_2 \buildrel{d_2}\over\longrightarrow F_1 \buildrel{d_1}\over\longrightarrow R_0 
\end{equation}
on the free $R_0$-modules $F_1$,$F_2$,$F_3$ having bases $\lbrace e_1, \ldots, e_n \rbrace$, $\lbrace f_1, \ldots, f_{n} \rbrace$, $\lbrace g \rbrace$ and differentials
$$d_3 =\bmatrix 0  \\ \vdots  \\ 0 \\ 1 \endbmatrix , 
d_2=\bmatrix 1 & 0 &  \ldots & 0 & 0\\  0 & 1 & \ldots &  0 & 0\\ \vdots & \vdots & \ddots & \vdots & \vdots \\ 0 & 0 & \ldots &  1 & 0 \\
0 & 0 &  \ldots & 0 & 0 \endbmatrix , 
d_1=\bmatrix 0 & 0 & \ldots &1 \endbmatrix. $$

We are going to prove the following theorem, expressed using the notation of Definition \ref{pfaffians}.

\begin{thm}
\label{thsplitn}
Let $n=2s \geq 4$. 
The differentials of the complex $\FF^{top}_\bullet$ associated to the complex (\ref{splitcomplexn}) are:
$$ v^{(1)}_{top} = \bmatrix P_{\hat{1},\hat{n}} & -P_{\hat{2},\hat{n}} & \ldots & P_{\hat{n-1},\hat{n}} & 0  
 \endbmatrix, $$ 
 $$ v^{(2)}_{top}= \bmatrix 
0 & -b_{12} &  -b_{13} & \ldots & -b_{1,n-1} & 0 \\  
b_{12} & 0 & -b_{23} & \ldots & -b_{2,n-1} & 0 \\ 
b_{13} & b_{23} & 0 & \ldots & -b_{3,n-1} & 0 \\
\vdots & \vdots & \vdots & \ddots & \vdots & \vdots \\
b_{1,n-1} & b_{2,n-1} & b_{3,n-1} & \ldots & 0 & 0  \\ 
b_{1n} & b_{2n} & b_{3n} & \ldots & b_{n-1,n}  &  1  \\
 \endbmatrix,
 v^{(3)}_{top}= \bmatrix -P_{\hat{1},\hat{n}} \\ P_{\hat{2},\hat{n}} \\ \vdots \\ -P_{\hat{n-1},\hat{n}} \\ P
 \endbmatrix.  $$
 Moreover, the complex $\FF^{top}_\bullet$ is acyclic.
\end{thm}

\begin{proof} (of acyclicity)
It is easy to check that the composition of the differentials are zero. To show acyclicity one may observe that the maximal minors of  $ v^{(2)}_{top} $ involving the last row and column are equal to maximal minors of the skew-symmetric matrix on the variables $b_{ij}$ with $i,j \leq n-1$. This shows that the rank of $v^{(2)}_{top}$ is $n-1$ and that the depth of its ideal of maximal minors is at least 2. Since $n \geq 4$, there are at least three distinct submaximal pfaffians of $B$ among the entries of $v^{(3)}_{top}$. They form a regular sequence in $R$ by results in \cite{CVW18}. Therefore the complex is acyclic by the Buchsbaum-Eisenbud criterion \cite{BE73}.
\end{proof}

Let us start the computation of the maps $v^{(i)}_j$ introduced in Section 2. Observe that the expressions $G_{ij}$ introduced in (\ref{bvar}) are simply $G_{ij} = b_{ij} g$, $d(G_{ij}) = b_{ij} f_n$ (for this format we do not need the notation $b_{ij}^{k}$ since the rank of $F_3$ is one). Furthermore, for $i \leq n-1$, $d(e_i)=0$ and $d(f_i)=e_i$ while $d(e_n)=1$ and $d(f_n)=0$.

Using Leibniz rule and exactness of $\FF$ one can compute the maps $v^{(3)}_1$,$v^{(2)}_1$,$v^{(1)}_1$ corresponding to multiplicative structure as:
$$e_i^.e_j=  b_{ij} f_n, \mbox{ for } 1 \leq i < j \leq n-1, \, \, \, 
 e_i^.e_n = -f_i + b_{in} f_n. $$
$$e_i^.f_h=  -b_{ih}g, \mbox{ for } i \leq n-1, \, h \leq n-1, \, \, \, e_n^.f_h=  b_{hn}g, \mbox{ for } h \leq n-1, $$
$$e_i^.f_n=  0, \mbox{ for } i \leq n-1, \, \, \, e_n^.f_n=  g. $$
$$e_i^.e_j^.e_k =  0, \mbox{ for } 1 \leq i < j < k \leq n-1, \, \, 
 e_i^.e_j^.e_n = (-1)^{i+j+1}b_{ij}g. $$
 
 For this format the map $v^{(2)}_1: F_1 \otimes F_2 \to F_3$ is already equal to $v^{(2)}_{top}$.
For $s \geq 2$, the maps $v^{(3)}_s,v^{(1)}_s$ can be computed by lifting the Bruns cycles described in \cite[Section 11]{W18}.

 The map $v^{(3)}_2: \bigwedge^4 F_1 \to F_2 \otimes F_3$ is calculated as lifting of the cycle $q$ in the following complex:
 \begin{equation}
 \label{v3,2}
 \begin{matrix} 0&\rightarrow& \bigwedge^2F_3 &\rightarrow& F_3\otimes F_2 &\rightarrow  S_2 F_2 &\rightarrow S_2F_1 \\ &&&& \uparrow\ v^{(3)}_2 & \nearrow q  \\&&&&  \bigwedge^4F_1   \end{matrix}
 \end{equation}
 where we set $ q(e_i,e_j,e_k,e_r) = e_i^.e_j \otimes e_k^.e_r - e_i^.e_k \otimes e_j^.e_r + e_i^.e_r \otimes e_j^.e_k. $
 
 An easy computation shows that the image of $q$ is in the kernel of $d_2 \otimes d_2$, and therefore in  the image of $d_3 \otimes 1: F_3 \otimes F_2 \to S_2F_2$. Hence the image of $q$ can be lifted to $F_3 \otimes F_2$. This lifting is in general not unique, but in our case it is since $\bigwedge^2 F_3 = 0$ and $d_3 \otimes 1$ is injective.
 To compute this map concretely, we use the notation of Definition \ref{pfaffians}. 
  In the case $i,j,k,r \leq n-1$, we get
 $$ q(e_i,e_j,e_k,e_r) = b_{ij} f_n \otimes b_{kr} f_n - b_{ik} f_n \otimes b_{jr} f_n + b_{ir} f_n \otimes b_{jk} f_n  = P_{ijkr} (f_n \otimes f_n). $$ Using that $f_n=d(g)$ we could easily lift this to $ F_3 \otimes F_2 $ getting 
 $$ v_{2}^{(3)}(e_i,e_j,e_k,e_r) = P_{ijkr}(f_n \otimes g).  $$
 In the case $r=n$ some extra terms appear in the result:
 $$ q(e_i,e_j,e_k,e_n) = b_{ij} f_n \otimes (-f_k + b_{kn} f_n) - b_{ik} f_n \otimes (-f_j + b_{jn} f_n) +   b_{jk} f_n \otimes (-f_i + b_{in} f_n) = $$ 
 $$ = P_{ijkr} (f_n \otimes f_n) - b_{ij} \otimes (f_k \otimes f_n) + b_{ik} \otimes (f_j \otimes f_n) - b_{jk} \otimes (f_i \otimes f_n). $$
 Hence $$ v_{2}^{(3)}(e_i,e_j,e_k,e_n) = P_{ijkn}(f_n \otimes g) - b_{ij} \otimes (f_k \otimes g) + b_{ik} \otimes (f_j \otimes g) - b_{jk} \otimes (f_i \otimes g).  $$

  \begin{remark}
 When $n=4$ the top maps are $v^{(3)}_2$,$v^{(2)}_1$,$v^{(1)}_1$ and therefore the differentials of the complex $\FF^{top}_\bullet$ are represented by the matrices $$v^{(3)}_{top} =\bmatrix -b_{23}  \\ b_{13}  \\ -b_{12} \\ P \endbmatrix , 
v^{(2)}_{top} =\bmatrix 0 & -b_{12} &  -b_{13} & 0 \\  b_{12} & 0 & -b_{23} &  0 \\ b_{13} & b_{23} & 0 & 0  \\ b_{14} & b_{24} & b_{34} &  1  \\
 \endbmatrix , 
v^{(1)}_{top} =\bmatrix b_{23} & -b_{13} &  b_{12} & 0 \endbmatrix. $$ Here $P=b_{12}b_{34}-b_{13}b_{24}+b_{14}b_{23}$. One can easily show manually that the complex $\FF^{top}_\bullet$ having these differentials is acyclic.
 \end{remark}
 
For $n > 4$, the map $v^{(3)}_{top}$ can be computed in the following way.

\begin{prop}
\label{splitn3}
For $n=2s$ the map $v^{(3)}_{top} = v^{(3)}_{s}: \bigwedge^{n} F_1 \to S_{s-1} F_3 \otimes F_2  $ is obtained by lifting the image of the map $v^{(3)}_{1} \otimes v^{(3)}_{s-1}$. In particular
$$ v_{top}^{(3)}(e_{1},\ldots,e_{n}) = P(f_n \otimes g^{s-1}) + \sum_{i=1}^{n-1}(-1)^{i}P_{\hat{i},\hat{n}}(f_i \otimes g^{s-1}).  $$
\end{prop}

\begin{proof}
As we did in the case $s=2$, we write $$q: \bigwedge^n F_1 \to \bigwedge^2 F_1 \otimes \bigwedge^{n-2} F_1 \buildrel{ v^{(3)}_{s-1} \otimes v^{(3)}_{1}}\over\longrightarrow S_{s-2} F_3 \otimes F_2 \otimes F_2. $$
Explicitly 
$$q(e_1, \ldots, e_n)= \sum_{i=1}^{n-1} (-1)^{i+1} e_i^.e_n \otimes v_{s-1}^{(3)}(\hat{e_i},\hat{e_n}).   $$
We can work by induction on $n$ since we already computed $v^{(3)}_{1}$ and $v^{(3)}_{2}$.
We know that $e_i^.e_n = -f_i + b_{in}f_n$ and,
looking at those maps and computing the next ones iteratively, is easy to see that the term 
$v_{s-1}^{(3)}(\hat{e_i},\hat{e_n})= P_{\hat{i},\hat{n}}(f_n \otimes g^{s-2})$.  
It follows that the component of $ q(e_1, \ldots, e_n) $ with respect to $f_n \otimes f_n \otimes g^{s-2}$ is given by $$ \sum_{i=1}^{n-1} (-1)^{i+1} b_{in}  P_{\hat{i},\hat{n}} = P.  $$
For $i \neq n$, the component of $ q(e_1, \ldots, e_n) $ with respect to $f_i \otimes f_n \otimes g^{s-2}$ is given by 
$  (-1)^{i}P_{\hat{i},\hat{n}}. $
Using the fact that $d(g)=f_n$ and lifting $q$ to $ S_{s-1} F_3 \otimes F_2 $ we get the desired result.
\end{proof}
 
 Knowing the maps associated to graded components of $W(d_3)$ allows to compute easily those associated to $W(d_1)$.

 \begin{prop}
 \label{splitn1}
For $n=2s$ the map $v^{(1)}_{top} = v^{(1)}_{s-1}: \bigwedge^{n-1} F_1 \to S_{s-1} F_3  $ is obtained by lifting the image of the map $v^{(1)}_{0} \otimes v^{(3)}_{s-1}$. For $i \neq n$,
$ v_{top}^{(1)}(\hat{e_{i}}) = (-1)^{i+1}P_{\hat{i}, \hat{n}}g^{s-1}, \mbox{ while }  v_{top}^{(1)}(\hat{e_{n}}) = 0. $
\end{prop}

\begin{proof}
Set $$q: \bigwedge^{n-1} F_1 \to  F_1 \otimes \bigwedge^{n-2} F_1 \buildrel{ v^{(1)}_{0} \otimes v^{(3)}_{s-1}}\over\longrightarrow S_{s-2} F_3 \otimes F_2. $$ Explicitly 
$$ q(\hat{e_{i}}) = \sum_{j \neq i}(-1)^{i+j}d(e_j)v^{(3)}_{s-1}(\hat{e_i},\hat{e_j}) . $$ The result follows now from the fact that $d(e_n)=1$ and $d(e_j)=0$ for every $j \neq n$. 
\end{proof}

Combining all the preceding results we get the description of the differential of the complex $\FF^{top}_\bullet$ of Theorem \ref{thsplitn}.

 
 
 

\subsection{Direct sum of a generic Hilbert-Burch complex and an identity map}

Here we compute the differentials of the complex presented in Theorem \ref{thhbn}.
All the computation of maps $v^{(i)}_j$ is done following the same method used for the split exact case of the same format. Since this case is more complicated, we skip several parts of the long computation and only highlight the key details. 

Consider the same setting used in Section 1.1 to state Theorem \ref{thhbn}. In particular we let $v_0^{(3)},v_0^{(2)},v_0^{(1)}$ be the differentials of the complex (\ref{hilb1nn1}).


To proceed with the computation we need to use some standard relation involving minors of the generic $n \times n-1$ matrix $X$. 

\begin{lem}
\label{minorlemma}
For $1 \leq i < j \leq n$, the result of the multiplication map $v^{(3)}_{1}$ is obtained as
\begin{equation}
\label{eqmult}
e_i^.e_j=  \sum_{k=1}^{n-1}(-1)^{i+j+k} X_{\hat{k}}^{\hat{i},\hat{j}} f_k + b_{ij}f_n.
\end{equation}
\end{lem}

\begin{proof}
Set $\alpha_{ij}= \sum_{k=1}^{n-1}(-1)^{i+j+k} X_{\hat{k}}^{\hat{i},\hat{j}} f_k $. We need to show that $d_2(\alpha_j) = d(e_i)e_j-d(e_j)e_i = Y_i e_j - Y_j e_i$. Using that $d(f_k)= \sum_{h=1}^n X_{hk} e_h$ we get
$$ d_2(\alpha_{_{ij}}) = \sum_{k=1}^{n-1}(-1)^{i+j+k} X_{\hat{k}}^{\hat{i},\hat{j}} \sum_{h=1}^n X_{hk} e_h = 
(-1)^{i+j}\sum_{h=1}^n \left( \sum_{k=1}^{n-1}(-1)^{k} X_{\hat{k}}^{\hat{i},\hat{j}}X_{hk} \right) e_h $$
Now, if $h=i$ (or $h=j$) the coefficient inside parenthesis is $\pm Y_j$ (or $\mp Y_i$ ), otherwise it is zero. 
\end{proof}

 Keeping the notation $e_i^.e_j= \alpha_{ij} + b_{ij}f_n$ introduced in the previous proof, after a similar computation involving relations between minors, we have the following lemma.
 
 \begin{lem}
 \label{relat1}  For every $1 \leq i \leq n $ and $1 \leq h \leq n-1$, the following relation holds:
\begin{equation} \label{rel1}
Y_i f_h - \sum_{j=1}^n X_{jh} \alpha_{ij} = 0
\end{equation}
\end{lem}

This allow us to compute the multiplication map $v^{(2)}_{top} = v^{(2)}_{1}$. 

\begin{prop}
 \label{hbn2}
The multiplication $v^{(2)}_{1}$ is computed as
$$ e_i^.f_h = \left\{ \begin{array}{cc}  -\sum_{j=1}^n X_{jh}b_{ji} g    &\mbox{if } h \neq n\\
                                 Y_ig            &\mbox{if } h=n  \\
\end{array}\right.  $$
\end{prop}

\begin{proof}
If $h \neq n$, we have 
$$ d(e_i^.f_h) = Y_if_h - e_i\sum_{j=1}^n X_{jh} e_j = Y_if_h - \sum_{j=1}^n X_{jh} (\alpha_{ij} + b_{ij} f_n ). $$
Using Lemma \ref{relat1}, this is equal to $-\sum_{j=1}^n X_{jh}b_{ji}f_n$ and can be lifted to $-\sum_{j=1}^n X_{jh}b_{ji}g$.
If $h=n$, $ d(e_i^.f_n) = Y_if_n  $ and therefore $ e_i^.f_n = Y_ig.  $
\end{proof}

We compute inductively the maps corresponding to the top graded component of $W(d_3)$ and $W(d_1)$ as  in Propositions \ref{splitn3} and \ref{splitn1}. We use the same notation of Definition \ref{pfaffians} to denote pfaffians of the generic skew-symmetric matrix in the variables $b_{ij}$.

\begin{prop}
\label{hbn3}
For $n=2s$ the map
$$ v_{top}^{(3)}(e_{1},\ldots,e_{n}) = P(f_n \otimes g^{s-1}) + \sum_{k=1}^{n-1}(-1)^{k}\left( \sum_{1 \leq i< j \leq n} (-1)^{i+j} X_{\hat{k}}^{\hat{i},\hat{j}} P_{\hat{i},\hat{j}} \right) (f_k \otimes g^{s-1}).  $$
\end{prop}

\begin{proof}
Similarly to the case of  the split exact complex, we compute $v^{(3)}_{top} = v^{(3)}_{s}: \bigwedge^{n} F_1 \to S_{s-1} F_3 \otimes F_2  $ as lifting of the map 
$q: \bigwedge^n F_1 \to S_{s-2} F_3 \otimes S_2F_2. $
Consider the exact complex 
\begin{equation}
 0 \longrightarrow S_{s-1} F_3 \otimes F_2  \buildrel{d_3 \otimes 1}\over\longrightarrow  S_{s-2} F_3 \otimes S_2F_2 \buildrel{1 \otimes d_2 \otimes d_2}\over\longrightarrow S_{s-2} F_3 \otimes S_2F_1 
\end{equation}
It is clear by the computation of Bruns cycles \cite{W18} that the image of $q$ is contained in the kernel of the map $1 \otimes d_2 \otimes d_2$ and therefore in the image of $d_3 \otimes 1$. 

Working by induction on $n$, for distinct $i,j$ consider the term $v_{s-1}^{(3)}(\hat{e_i},\hat{e_j}) \in S_{s-2} F_3 \otimes F_2$. This term can be expressed as $\delta_{\hat{i},\hat{j}} + \beta_{\hat{i},\hat{j}}(g^{s-2} \otimes f_n)$ where $\delta_{\hat{i},\hat{j}}$ is the sum of all the components with respect to vectors $(g^{s-1} \otimes f_h)$ for $h \neq n.$ 

Let us write 
$$q(e_1, \ldots, e_n)= \sum_{i=1}^{n-1} (-1)^{i+1} e_i^.e_n \otimes v_{s-1}^{(3)}(\hat{e_i},\hat{e_n}) =  \sum_{i=1}^{n-1} (-1)^{i+1} (\alpha_{in} + b_{in}f_n) \otimes (\delta_{\hat{i},\hat{n}} + \beta_{\hat{i},\hat{n}}(g^{s-2} \otimes f_n)). $$
Since for $h,k \neq n$, $g^{s-2} \otimes d(f_h) \otimes d(f_k)$ is not zero, but the the image of $q$ is contained in the kernel of $1 \otimes d_2 \otimes d_2$, we must have that the components with respect to all those vectors are zero, and hence 
$\sum_{i=1}^{n-1} (-1)^{i+1} \alpha_{in}  \otimes \delta_{\hat{i},\hat{n}} = 0$.
It follows that 
$$q(e_1, \ldots, e_n)=  \sum_{i=1}^{n-1} (-1)^{i+1} \alpha_{in} \otimes \beta_{\hat{i},\hat{n}}(g^{s-2} \otimes f_n) + b_{in}f_n \otimes \delta_{\hat{i},\hat{n}} + b_{in}f_n \otimes \beta_{\hat{i},\hat{n}}(g^{s-2} \otimes f_n)). $$

Now, the component of $ q$ with respect to $ f_n \otimes f_n \otimes g^{s-2} $ is  $ \sum_{i=1}^{n-1} (-1)^{i+1} b_{in}  \beta_{\hat{i},\hat{n}}.  $ By induction, this implies that $\beta_{\hat{i},\hat{n}} = P_{\hat{i},\hat{n}}$ and the component of $q$ along $ f_n \otimes f_n \otimes g^{s-2} $ is equal to $P$.

Instead, for $k \neq n$, the component of $ q $ with respect to $f_k \otimes f_n \otimes g^{s-2}$ is computed inductively to be 
$$  \sum_{1 \leq i<j \leq n} (-1)^{i+j+k} X_{\hat{k}}^{\hat{i},\hat{j}} P_{\hat{i},\hat{j}}.   $$
The thesis is now obtained by lifting $q$ to $ S_{s-1} F_3 \otimes F_2 $.
\end{proof}

 

 \begin{prop}
 \label{hbn1}
For $n=2s$ the map
$ v_{top}^{(1)}(\hat{e_{i}}) = \sum_{j \neq i}(-1)^{i+j}Y_j P_{\hat{i},\hat{j}}g^{s-1}. $
\end{prop}

\begin{proof}
Set $q$ as in Proposition \ref{splitn1} and look at the exact complex
\begin{equation}
 0 \longrightarrow S_{s-1} F_3  \buildrel{1 \otimes d_3}\over\longrightarrow  S_{s-2} F_3 \otimes F_2 \buildrel{1 \otimes d_2}\over\longrightarrow S_{s-2} F_3 \otimes F_1. 
\end{equation}
Also in this case, the image of $q$ is contained in the kernel of $1 \otimes d_2$ and hence in the image of 
$1 \otimes d_3$. We have
$$ q(\hat{e_{i}}) = \sum_{j \neq i}(-1)^{i+j}Y_j v^{(3)}_{s-1}(\hat{e_i},\hat{e_j}) = \sum_{j \neq i}(-1)^{i+j}Y_j (\delta_{\hat{i},\hat{j}} + \beta_{\hat{i},\hat{j}}(g^{s-2} \otimes f_n)) . $$ But now, since $q$ is in the kernel of $1 \otimes d_2$, we must have $ \sum_{j \neq i}(-1)^{i+j}Y_j \delta_{\hat{i},\hat{j}} = 0$ and thus 
 $   q(\hat{e_{i}}) = \sum_{j \neq i}(-1)^{i+j}Y_j P_{\hat{i},\hat{j}}(g^{s-2} \otimes f_n).     $ The thesis follows by lifting.
\end{proof}

Putting together Propositions \ref{hbn2},\ref{hbn3} and \ref{hbn1} and computing the explicit entries of the maps one gets the differentials described in Theorem \ref{thhbn}.

\section{Format $(1,4,m+3,m)$}\label{sec:14nn-3}

In this section we compute complexes analogous to the ones of the previous section for the format $(1,4,m+3,m)$. The main result for the split exact case is presented in Theorem \ref{thsplitm} and the required computations will be described after it. In the second part of the section we sketch the computation of the complex described in Theorem \ref{thhbm}.

\subsection{Split exact complex}
In the following $R_0$ can be taken to be a field of characteristic zero.
Let $m \geq 1$ and consider the free resolution 
\begin{equation}
\label{split14m+3m}
\FF: 0 \longrightarrow F_3 \buildrel{d_3}\over\longrightarrow  F_2 \buildrel{d_2}\over\longrightarrow F_1 \buildrel{d_1}\over\longrightarrow R_0 
\end{equation}
on the free $R_0$-modules $F_1$,$F_2$,$F_3$ having bases $\lbrace e_1, e_2, e_3, e_4\rbrace$, $\lbrace f_1, \ldots, f_{m+3} \rbrace$, $\lbrace g_1, \ldots, g_m\rbrace$ and differentials
$$d_3 =\bmatrix 0 & \ldots & 0 \\ 0 & \ldots  & 0 \\ 0 & \ldots & 0 \\ & I_n & \endbmatrix , d_2=\bmatrix 1 & 0 & 0 & 0 & \ldots  & 0\\  0 & 1 & 0 &0 & \ldots  & 0\\ 0 & 0 & 1 &0 & \ldots  & 0\\ 0 & 0 & 0 &0 & \ldots  & 0 \endbmatrix , d_1=\bmatrix 0 & 0 & 0 & 1 \endbmatrix. $$

In the statement of the theorem that we prove in this section we take the same notation of Definition \ref{minors-pfaffians}, and, moreover,
we denote by $$ L_{s,t}:= b_{12}^sb_{34}^t - b_{13}^sb_{24}^t+ b_{23}^sb_{14}^t, $$ for $1 \leq s,t \leq m$ certain mixed-pfaffians in the variables $b_{ij}$.


\begin{thm}
\label{thsplitm}
Let $m \geq 4$. The complex $\FF^{top}_\bullet$ associated to the complex (\ref{split14m+3m}) is the minimal free resolution of the ideal $J_m$ generated by the entries of $v^{(1)}_{top}$. If $m=2s$ is even
$$J_{2s}= \left( \sum_{u,t}B^{ut}_{12,13}\Gamma_{\hat{u},\hat{t}}, \sum_{u,t}B^{ut}_{12,23}\Gamma_{\hat{u},\hat{t}}, \sum_{u,t}B^{ut}_{13,23}\Gamma_{\hat{u},\hat{t}}, \Gamma \right),$$ 
If $m=2s+1$ is odd
$$J_{2s+1}= \left( \sum_{u}b^{u}_{23}\Gamma_{\hat{u}}, \sum_{u}b^{u}_{13}\Gamma_{\hat{u}}, \sum_{u}b^{u}_{12}\Gamma_{\hat{u}}, \sum_{u,t,v}B^{utv}_{12,13,23}\Gamma_{\hat{u},\hat{t},\hat{v}} \right).$$ In both ideals the sums are taken over all the required $1 \leq u < t < v \leq m.$
The other differential of the complex $\FF^{top}_\bullet$ are:
$$    v^{(2)}_{top}= \bmatrix 
\lbrace e_1,f_1 \rbrace & \lbrace e_1,f_2 \rbrace & \ldots & \lbrace e_1,f_{m+3} \rbrace   \\ 
 \lbrace e_2,f_1 \rbrace & \lbrace e_2,f_2 \rbrace & \ldots & \lbrace e_2,f_{m+3} \rbrace  \\ 
\lbrace e_3,f_1 \rbrace & \lbrace e_3,f_2  \rbrace & \ldots & \lbrace e_3,f_{m+3} \rbrace  \\ 
\lbrace e_4,f_1 \rbrace & \lbrace e_4,f_2 \rbrace  & \ldots & \lbrace e_4,f_{m+3} \rbrace  \\ 
 \endbmatrix,
 v^{(3)}_{top}= \bmatrix 
 -b_{23}^1 & -b_{23}^2 & \ldots  & -b_{23}^m
 \\  b_{13}^1 & b_{13}^2 & \ldots  & b_{13}^m
 \\ -b_{12}^1 & -b_{12}^2 & \ldots  & -b_{12}^m
 \\ L_{1,1} & L_{1,2}+c_{12}  & \ldots  & L_{1,m}+c_{1m} 
 \\ \vdots & \vdots  & \ldots  & \vdots
 \\  L_{m,1}-c_{1m}    &  L_{m,2}-c_{2m}   &   \ldots  & L_{m,m}
 \endbmatrix,  $$
  where, for $i,j,k=1,2,3$, $h\geq 4$ and $1 \leq u < t < v < w \leq m$, if $m$ is even:
 
$$ \lbrace e_i,f_i \rbrace =  \Gamma +  \sum_{u,t}B^{ut}_{i4,jk}\Gamma_{\hat{u},\hat{t}}, 
\mbox{  } \lbrace e_i,f_j \rbrace =  \sum_{u,t}B^{ut}_{j4,jk}\Gamma_{\hat{u},\hat{t}}, \mbox{  }  
\lbrace e_i,f_h \rbrace =  \sum_{u \neq h-3}b^{u}_{jk}\Gamma_{\hat{u},\widehat{h-3}}, $$ 
$$ \lbrace e_4,f_j \rbrace =  \sum_{u,t}B^{ut}_{ij,jk}\Gamma_{\hat{u},\hat{t}}  + \sum_{u,t,v,w}B^{utvw}_{12,13,23,j4}\Gamma_{\hat{u},\hat{t},\hat{v},\hat{w}}, 
\mbox{  }  \lbrace e_4,f_h \rbrace = \sum_{u,t,v \neq h-3}B^{utv}_{12,13,23}\Gamma_{\hat{u},\hat{t},\hat{v},\widehat{h-3}}.  $$ 
 If instead $m$ is odd:  
$$ \lbrace e_i,f_i \rbrace =   \sum_{u,t,v}B^{utv}_{i4,ij,ik}\Gamma_{\hat{u},\hat{t},\hat{v}}, 
\mbox{  } \lbrace e_i,f_j  \rbrace=  \sum_{u,t,v}B^{utv}_{j4,ij,ik}\Gamma_{\hat{u},\hat{t},\hat{v}} - \sum_{u}b_{ij}^u \Gamma_{u}, 
$$  $$  
\lbrace e_i,f_h \rbrace =  \sum_{u,t \neq h-3}B^{ut}_{ij,ik}\Gamma_{\hat{u},\hat{t},\widehat{h-3}}, \mbox{  } \lbrace e_4,f_j \rbrace =  \sum_{u}b_{j4}^u \Gamma_{u}, 
\mbox{  } \lbrace e_4,f_h \rbrace = \Gamma_{h-3}.  $$ 
Moreover, $\FF^{top}_\bullet$ is acyclic.
\end{thm}


\begin{example}
In the case $m=2$, the complex $\FF^{top}_\bullet$ is the minimal free resolution of the ideal $ J_5= (b_{13}^1 b_{12}^2 - b_{13}^2 b_{12}^1, b_{12}^1 b_{23}^2 - b_{12}^2 b_{23}^1, b_{23}^1 b_{13}^2 - b_{23}^2 b_{13}^1, c_{12}).  $ The other top maps are 
$$ \footnotesize v^{(2)}_{top}= \bmatrix 
b_{14}^1 b_{23}^2 - b_{14}^2 b_{23}^1 + c_{12} & b_{24}^1 b_{23}^2 - b_{24}^2 b_{23}^1  & b_{34}^1 b_{23}^2 - b_{34}^2 b_{23}^1 & b_{23}^2 & -b_{23}^1   \\ 
 b_{14}^1 b_{13}^2 - b_{14}^2 b_{13}^1 & b_{24}^1 b_{13}^2 - b_{24}^2 b_{13}^1 + c_{12} & b_{34}^1 b_{13}^2 - b_{34}^2 b_{13}^1 & -b_{13}^2 & b_{13}^1   \\ 
b_{14}^1 b_{12}^2 - b_{14}^2 b_{12}^1 & b_{24}^1 b_{12}^2 - b_{24}^2 b_{12}^1   & b_{34}^1 b_{12}^2 - b_{34}^2 b_{12}^1 + c_{12} & b_{12}^2 & -b_{12}^1   \\ 
b_{12}^1 b_{13}^2 - b_{12}^2 b_{13}^1 & b_{23}^1 b_{12}^2 - b_{23}^2 b_{12}^1   & b_{13}^1 b_{23}^2 - b_{13}^2 b_{23}^1 & 0 & 0  \\ 
 \endbmatrix
 $$  $$ \footnotesize
 v^{(3)}_{top}= \bmatrix 
 -b_{23}^1 & -b_{23}^2 
 \\  b_{13}^1 & b_{13}^2 
 \\ -b_{12}^1 & -b_{12}^2 
 \\ b_{12}^1b_{34}^1 - b_{13}^1b_{24}^1+ b_{14}^1b_{23}^1   & 
b_{12}^2b_{34}^1 - b_{13}^2b_{24}^1+ b_{23}^2b_{14}^1 + c_{12}   
 \\  b_{12}^1b_{34}^2 - b_{13}^1b_{24}^2+ b_{23}^1b_{14}^2 - c_{12}  &  
 b_{12}^2b_{34}^2 - b_{13}^2b_{24}^2+ b_{14}^2b_{23}^2  
 \endbmatrix.  
   $$
\end{example}

We begin our computation by looking at the maps $v^{(3)}_1, v^{(2)}_1, v^{(1)}_1$, as done for the format $(1,n,n,1)$. We get
$$e_i^.e_j=  d(G_{ij}), \mbox{ for } 1 \leq i < j \leq 3, \, \, \, 
 e_i^.e_4 = -f_i + d(G_{i4}). $$
$$e_i^.f_h=  -G_{ih}, \mbox{ for } i \leq 3, \, h \leq 3, \, \, \, e_4^.f_h=  G_{h4}, \mbox{ for } h \leq 3, $$
$$e_i^.f_h=  0, \mbox{ for } i \leq 3, \, h \geq 4, \, \, \, e_4^.f_h=  g_{h-3}, \mbox{ for } h \geq 4. $$
$$e_i^.e_j^.e_k =  0, \mbox{ for } 1 \leq i < j < k \leq 3, \, \, \, 
 e_i^.e_j^.e_4 = (-1)^{i+j}G_{ij}. $$

 For this format the top component of $W(d_3)$ is already the map $v^{(3)}_{2}$. Recalling the definition of the sets of defect variables given in Definitions \ref{bvar} and \ref{cvar}, we compute that map as before using the relations in (\ref{v3,2}).
Hence, setting $\varepsilon:= e_1 \wedge e_2 \wedge e_3 \wedge e_4$, we have $$ v^{(3)}_{top}(\varepsilon)= v^{(3)}_{2}(\varepsilon)=  G_{12} \otimes e_3^.e_4 - G_{13} \otimes e_2^.e_4 +  G_{23} \otimes e_1^.e_4 + d(C). $$

For arbitrary $n \geq 2$, we first compute the maps $v^{(1)}_{n}$ corresponding to graded components of $W(d_1)$. 
For $n=2s$, $v^{(1)}_{n}: ( \bigwedge^4 F_1 )^{\otimes s} \otimes F_1 \to \bigwedge^n F_3$, while for $n=2s+1$ we have 
$v^{(1)}_{n}: ( \bigwedge^4 F_1 )^{\otimes s} \otimes \bigwedge^3 F_1 \to \bigwedge^n F_3.$

All those maps are obtained by lifting the image of $ q^{(1)}_{n} = v^{(1)}_{n-1} \otimes v^{(3)}_{1} - v^{(1)}_{n-2} \wedge v^{(3)}_{2} $ and hence, 
starting from $v^{(1)}_{1}$,$v^{(1)}_{2}$, is possible to compute inductively all these higher maps. 
This will appear in terms of the elements $C^{(s)}$ introduced in Definition \ref{highC}.

\begin{prop}
\label{splitm1}
 The map $v^{(1)}_{n}$ is obtained by lifting to $\bigwedge^n F_3$ the image of the map $ q^{(1)}_{n}$. Let $i,j,k = 1,2,3$.
For $n=2s$, $$  v^{(1)}_{n}(e_i) = (-1)^{i} G_{ij} \wedge G_{ik} \wedge C^{(s-1)} \mbox{ and }  v^{(1)}_{n}(e_4) =  C^{(s)}. $$
For $n=2s+1$, $$  v^{(1)}_{n}(e_i) =  G_{jk} \wedge C^{(s)} \mbox{ and }  v^{(1)}_{n}(e_4) =  G_{12} \wedge G_{13} \wedge G_{23} \wedge C^{(s-1)}. $$
In particular for every $n \geq 4$, $ v^{(1)}_{n}(e_i) =  v^{(1)}_{n-2}(e_i) \wedge C.$
 \end{prop}

\begin{proof}
Since we already know $v^{(1)}_{0}= d_1$ and $v^{(1)}_{1}$, we can proceed by induction on $n$ and assume the thesis to be true for $n-2$ and $n-1$. Let us separate the even and the odd case. \\
\bf Case $n=2s$: \rm \\
Consider $ q^{(1)}_{n}: ( \bigwedge^4 F_1 )^{\otimes (s-1)} \otimes \bigwedge^4 F_1 \otimes F_1 \to \bigwedge^{n-1} F_3 \otimes F_2  $ defined as above. We show that the image of this map can be lifted to $\bigwedge^n F_3$.
It is sufficient to show this only for $e_1$ and $e_4$. We have
$$ q^{(1)}_{n}(e_1)=  v^{(1)}_{n-1}(e_2) \otimes e_1^.e_2 - v^{(1)}_{n-1}(e_3) \otimes e_1^.e_3 + v^{(1)}_{n-1}(e_4) \otimes e_1^.e_4 - v^{(1)}_{n-2}(e_1) \wedge v^{(3)}_{2}(\varepsilon) = $$
$$  G_{13} \wedge C^{(s-1)} \otimes d(G_{12}) - G_{12} \wedge C^{(s-1)} \otimes d(G_{13}) + G_{12} \wedge G_{13} \wedge G_{23} \wedge C^{(s-2)} \otimes e_1^.e_4 + $$ $$ - G_{13} \wedge G_{12} \wedge C^{(s-2)} \wedge [G_{12} \otimes e_3^.e_4 - G_{13} \otimes e_2^.e_4 +  G_{23} \otimes e_1^.e_4 + d(C)] = $$
$$ G_{13} \wedge C^{(s-1)} \otimes d(G_{12}) - G_{12} \wedge C^{(s-1)} \otimes d(G_{13}) - G_{13} \wedge G_{12} \wedge C^{(s-2)} \wedge d(C).  $$ By Definition \ref{highC}, this lifts to $G_{13} \wedge G_{12} \wedge C^{(s-1)}$. On the other hand:
$$ q^{(1)}_{n}(e_4)=  v^{(1)}_{n-1}(e_1) \otimes e_4^.e_1 - v^{(1)}_{n-1}(e_2) \otimes e_4^.e_2 + v^{(1)}_{n-1}(e_3) \otimes e_4^.e_3 - v^{(1)}_{n-2}(e_4) \wedge v^{(3)}_{2}(\varepsilon) = $$
$$  G_{23} \wedge C^{(s-1)} \otimes e_4^.e_1 - G_{13} \wedge C^{(s-1)} \otimes e_4^.e_2 + G_{12} \wedge C^{(s-1)} \otimes e_4^.e_3 + $$ $$ -  C^{(s-1)} \wedge [G_{12} \otimes e_3^.e_4 - G_{13} \otimes e_2^.e_4 +  G_{23} \otimes e_1^.e_4 + d(C)] =  -  C^{(s-1)} \wedge  d(C).  $$  This lifts to $C^{(s)}$. \\
\bf Case $n=2s+1$: \rm \\
Consider $ q^{(1)}_{n}: ( \bigwedge^4 F_1 )^{\otimes (s-1)} \otimes \bigwedge^4 F_1 \otimes \bigwedge^3 F_1 \to \bigwedge^{n-1} F_3 \otimes F_2   $. In this case one has to compute 
$$ q^{(1)}_{n}(e_1)=  v^{(1)}_{n-1}(e_2) \otimes e_3^.e_4 - v^{(1)}_{n-1}(e_3) \otimes e_2^.e_4 + v^{(1)}_{n-1}(e_4) \otimes e_2^.e_3 - v^{(1)}_{n-2}(e_1) \wedge v^{(3)}_{2}(\varepsilon)$$ and 
$$ q^{(1)}_{n}(e_4)=  v^{(1)}_{n-1}(e_1) \otimes e_2^.e_3 - v^{(1)}_{n-1}(e_2) \otimes e_1^.e_3 + v^{(1)}_{n-1}(e_3) \otimes e_1^.e_2 - v^{(1)}_{n-2}(e_4) \wedge v^{(3)}_{2}(\varepsilon) $$
and show in a similar way that images can be lifted to $\bigwedge^n F_3$. 
\end{proof}

The idea to compute all maps $v^{(2)}_{n}$ corresponding to components of $W(d_2)$ is quite similar. 
For $n=2s$, $ v^{(2)}_n: \bigwedge ^3 F_1 \otimes F_2 \cong (\bigwedge^4 F_1)^{\otimes s-1} \otimes \bigwedge ^3 F_1 \otimes F_2 \to \bigwedge^n F_3, $ while for $n=2s+1$
$ v^{(2)}_n: F_1 \otimes F_2 \cong (\bigwedge^4 F_1)^{\otimes s} \otimes F_1 \otimes F_2 \to \bigwedge^n F_3. $ 
The general way to construct these maps is to consider
$$ q_n :=  v^{(2)}_{n-1} \otimes v^{(3)}_1 -  v^{(2)}_{n-2} \wedge v^{(3)}_{2} + v^{(1)}_{n-1} \otimes 1_{F_2},  $$  and show that its image, contained in $\bigwedge^{n-1}F_3 \otimes F_2$, can be lifted to $ \bigwedge^{n}F_3 $.

\begin{prop}
\label{splitm2}
 The map $v^{(2)}_{n}$ is obtained by lifting to $\bigwedge^n F_3$ the image of the map $ q_{n}$. Let $i,j,k = 1,2,3$ and $h \geq 4$.
For $n=2s$, $$  v^{(2)}_{n}(e_i,f_i) =  C^{(s)} +G_{i4} \wedge G_{jk} \wedge C^{(s-1)}, \mbox{  }  v^{(2)}_{n}(e_i,f_j) =  G_{j4} \wedge G_{jk} \wedge C^{(s-1)}, \mbox{  } $$ $$  v^{(2)}_{n}(e_i,f_h) =  G_{jk} \wedge g_{h-3} \wedge C^{(s-1)}, $$ $$   
v^{(2)}_{n}(e_4,f_j) =  G_{ij} \wedge G_{jk} \wedge C^{(s-1)}+  G_{12} \wedge G_{13} \wedge G_{23} \wedge G_{i4} \wedge C^{(s-2)} , 
$$ $$ v^{(2)}_{n}(e_4,f_h) = G_{12} \wedge G_{13} \wedge G_{23} \wedge g_{h-3} \wedge C^{(s-2)}.  $$

For $n=2s+1$, $$  v^{(2)}_{n}(e_i,f_i) =  G_{i4} \wedge G_{ij} \wedge G_{ik} \wedge C^{(s-1)}, \mbox{  }  v^{(2)}_{n}(e_i,f_j) =  G_{j4} \wedge G_{ij} \wedge G_{ik} \wedge C^{(s-1)} - G_{ij} \wedge C^{(s)}, $$ $$   v^{(2)}_{n}(e_4,f_j) =  G_{j4} \wedge  C^{(s)}, \mbox{  }  v^{(2)}_{n}(e_i,f_h) = G_{ij} \wedge G_{ik} \wedge g_{h-3} \wedge C^{(s-1)}, \mbox{    } v^{(2)}_{n}(e_4,f_h) = g_{h-3} \wedge C^{(s)}.  $$
 \end{prop}

\begin{proof}
Again one can proceed by induction on $n$, assuming the thesis true for $n-2$ and $n-1$. Assume $\lbrace i,j,k,r \rbrace = \lbrace 1,2,3,4 \rbrace$. If $n$ is even, one has to compute
$$ q_n(e_i,f_h) =  v^{(1)}_{n-1}(e_i) \otimes f_h - v^{(2)}_{n-1}(e_r,f_h) \otimes e_j^.e_k + v^{(2)}_{n-1}(e_k,f_h) \otimes e_j^.e_r + $$ $$ - v^{(2)}_{n-1}(e_j,f_h) \otimes e_k^.e_r + v^{(2)}_{n-2}(e_1,f_h) \wedge v^{(3)}_{2}(\varepsilon),  $$ while if $n$ is odd one computes
$$ q_n(e_i,f_h) =  v^{(1)}_{n-1}(e_i) \otimes f_h - v^{(2)}_{n-1}(e_r,f_h) \otimes e_i^.e_r + v^{(2)}_{n-1}(e_k,f_h) \otimes e_i^.e_k + $$ $$ - v^{(2)}_{n-1}(e_j,f_h) \otimes e_i^.e_j + v^{(2)}_{n-2}(e_1,f_h) \wedge v^{(3)}_{2}(\varepsilon).  $$
The result is obtained considering separately all the cases with $i < 4$, $i=4$, $h \leq 3$ and $h \geq 4$, and by lifting the corresponding cycles, as it was done in Proposition \ref{splitm1}. We omit the computations since they require only a long and technical routine.
\end{proof}

Combining all the preceding results we get the description of the differential of the complex $\FF^{top}_\bullet$ of Theorem \ref{thsplitm}.




\subsection{Direct sum of a generic Hilbert-Burch complex and an identity map}

Here, we give a sketch of the computation of the differentials of the complex presented in Theorem \ref{thhbm}. We use the same method used for the split exact case to compute the maps $v^{(i)}_j$.
Consider the same setting used in Section 1.2 to state Theorem \ref{thhbm}. In particular we let $v_0^{(3)},v_0^{(2)},v_0^{(1)}$ be the differentials of the complex (\ref{hilb14m+3m}).

Using again Lemma \ref{minorlemma}, we obtain the multiplication map $v^{(3)}_1: \bigwedge^2 F_1 \to F_2$ as
\begin{equation}
e_i^.e_j=  (-1)^{i+j+1} \left( X_{\hat{1}}^{\hat{i}\hat{j}} f_1 - X_{\hat{2}}^{\hat{i}\hat{j}} f_2 + X_{\hat{3}}^{\hat{i}\hat{j}} f_3  \right) + d(G_{ij}).
\end{equation}
 
 To compute all the other maps, we need the following relations involving minors of $X$.

\begin{lem} 
\label{relbis}
The following relations hold for every choice of distinct $1 \leq i,j,k \leq 4 $ and $1 \leq h \leq 3$:
\begin{equation} \label{rel1bis}
Y_i f_h - X_{jh} (e_j^.e_i - d(G_{ji})) - X_{kh} (e_k^.e_i - d(G_{ki}))  - X_{lh} (e_l^.e_i - d(G_{li})) = 0
\end{equation}
\begin{equation} \label{rel2}
Y_k(e_i^.e_j - d(G_{ij})) - Y_j(e_i^.e_k - d(G_{ik})) + Y_i(e_j^.e_k - d(G_{jk})) = 0. 
\end{equation}
\end{lem}

\begin{proof}
Equation (\ref{rel1bis}) is exactly the same proved in Lemma \ref{relat1}. Equation (\ref{rel2}) can be easily proved showing that $d_2(Y_k(e_i^.e_j - d(G_{ij})) - Y_j(e_i^.e_k - d(G_{ik})) + Y_i(e_j^.e_k - d(G_{jk}))) = 0$ and hence this term can be lifted to $F_3$. Since all the terms involved in the equation have the coefficients with respect to $f_4, \ldots, f_{m+3}$ equal to zero, also their sum has to be zero. 
\end{proof}


Using the relation in Lemma \ref{relbis}, and working as for the split exact case, one can obtain $v^{(2)}_1$ and $v^{(1)}_1$ as follows: for $  h \leq 3 $
$e_i^.f_h= \sum_{r=1}^4 X_{rh}G_{ir},$ while, for $ h \geq 4$, $e_i^.f_h= Y_i g_{h-3}$. Furthermore, 
$e_i^.e_j^.e_k = Y_i G_{jk} - Y_j G_{ik} + Y_k G_{ij}.$ 

\medskip

Set $\varepsilon= e_1 \wedge e_2 \wedge e_3 \wedge e_4 $.
Computing $v^{(3)}_2$ using again the relations described in (\ref{v3,2}) we obtain 
$$ q(\varepsilon)= d(G_{12}) \otimes e_3^.e_4 - d(G_{13}) \otimes e_2^.e_4 + d(G_{14}) \otimes e_2^.e_3 +$$ $$ + d(G_{34}) \otimes e_1^.e_2 - d(G_{24}) \otimes e_1^.e_3 + d(G_{23}) \otimes e_1^.e_4 +$$
$$ - d(G_{12}) \otimes d(G_{34}) + d(G_{13}) \otimes d(G_{24}) - d(G_{14}) \otimes d(G_{23}). $$  


This term can be lifted in many different equivalent ways to the module $\bigwedge^2 F_3$ and
 $v^{(3)}_2 (\varepsilon)$ is given by any choice of such lift of $ q(\varepsilon) $ plus the term $d(C)$, defined in Definition \ref{cvar}. To present the map $v^{(3)}_{top} = v^{(3)}_2$, we choose the following lift: 
 
 \begin{prop}
 \label{hbm3}
 The map $v^{(3)}_{top}$ is obtained as 
 $$ v^{(3)}_2 (\varepsilon)=  G_{12} \otimes e_3^.e_4 - G_{13} \otimes e_2^.e_4 + G_{14} \otimes e_2^.e_3 + G_{34} \otimes e_1^.e_2 - G_{24} \otimes e_1^.e_3 + G_{23} \otimes e_1^.e_4 + $$
$$ + d(C)  - \frac{1}{2} [ G_{12} \otimes d(G_{34}) - G_{13} \otimes d(G_{24}) + G_{14} \otimes d(G_{23}) +  $$ $$  + d(G_{12}) \otimes G_{34} - d(G_{13}) \otimes G_{24} + d(G_{14}) \otimes G_{23}]. $$
 \end{prop}

\medskip

The maps $v^{(1)}_{n}$ are computed inductively as seen for the split exact case in Proposition \ref{splitm1}, but adjoining an extra term for any $n \geq 3$. 
  Indeed, now $v_2^{(3)}(\varepsilon)$ is not uniquely determined and such extra term represents this non-uniqueness and needs to be used in order to have a proper lift to $\bigwedge^{n} F_3$. This term is defined, if $n$ is odd, as $$\Gamma_n(e_i) :=  Y_i (G_{jk} \wedge G_{jl} \wedge G_{kl} \, \wedge C^{(s-2)}) \wedge [G_{il} \otimes e_j^.e_k - G_{ik} \otimes e_j^.e_l + G_{ij} \otimes e_k^.e_l] + $$ 
$$ Y_j (G_{kl} \wedge C^{(s-1)}) \wedge G_{ij}  \otimes e_k^.e_l  - Y_k (G_{jl} \wedge C^{(s-1)}) \wedge G_{ik} \otimes e_j^.e_l + Y_l (G_{jk} \wedge C^{(s-1)}) \wedge G_{il} \otimes e_j^.e_k]  $$ while if $n$ is even as $$ \Gamma_n(e_i) := Y_j (G_{ik} \wedge G_{il} \wedge C^{(s-2)}) \wedge (- G_{jl} \otimes e_i^.e_k  + G_{jk} \otimes e_i^.e_l) + $$ $$ -
   Y_k (G_{ij} \wedge G_{il} \wedge C^{(s-2)}) \wedge ( G_{kl} \otimes e_i^.e_j  + G_{jk} \otimes e_i^.e_l) + $$ $$  Y_l (G_{ij} \wedge G_{ik} \wedge C^{(s-2)}) \wedge ( G_{kl} \otimes e_i^.e_j  - G_{jl} \otimes e_i^.e_k). $$ 
 
 \begin{prop}
 \label{hbm1}
 Define $ \Gamma_n(e_i) $ as above and $$ q^{(1)}_{n} := v^{(1)}_{n-1} \otimes v^{(3)}_{1} - v^{(1)}_{n-2} \wedge v^{(3)}_{2}. $$ 
 The map $v^{(1)}_{n}$ is obtained by lifting to $\bigwedge^n F_3$ the image of the map $ q^{(1)}_{n}(e_i) + \Gamma_n(e_i), $ and
for $n=2s$, $$  v^{(1)}_{n}(e_i) = Y_i C^{(s)} + Y_j (G_{ik} \wedge G_{il} \wedge C^{(s-1)}) - Y_k (G_{ij} \wedge G_{il} \wedge C^{(s-1)}) + Y_l (G_{ij} \wedge G_{ik} \wedge C^{(s-1)}) $$ while 
for $n=2s+1$, $$  v^{(1)}_{n}(e_i) =  Y_i (G_{jk} \wedge G_{jl} \wedge G_{kl} \wedge C^{(s-1)}) + Y_j (G_{kl} \wedge C^{(s)}) - Y_k (G_{jl} \wedge C^{(s)}) + Y_l (G_{jk} \wedge C^{(s)}). $$ 
In particular for $n \geq 4$, $ v^{(1)}_{n}(e_i) =  v^{(1)}_{n-2}(e_i) \wedge C.$
 \end{prop}
 
\begin{proof}
This can be proved along the lines of Proposition \ref{splitm1}, with repeated application of relation (\ref{rel2}).
\end{proof}

Proceeding inductively  along the lines of Proposition \ref{splitm2}, and using repeatedly the relations in Lemma \ref{relbis}, all the maps $v^{(2)}_{n}$ can be described. 

\begin{prop}
\label{hbm2}
 The map $v^{(2)}_{n}$ is obtained by lifting to $\bigwedge^n F_3$ the image of the map $$ q_n :=  v^{(2)}_{n-1} \otimes v^{(3)}_1 -  v^{(2)}_{n-2} \wedge v^{(3)}_{2} + v^{(1)}_{n-1} \otimes 1_{F_2}. $$  Let $ \lbrace i,j,k,l \rbrace = \lbrace 1,2,3,4 \rbrace $ and suppose $j < k < l$.
  Then,
 for $n=2s$ we have, for $h \leq 3$:
$$ v^{(2)}_{n}(e_i \otimes f_h) = X_{ih} C^{(s)}+ X_{jh} (C^{(s-1)} \wedge G_{jl} \wedge G_{jk})+ $$ $$ - X_{kh} (C^{(s-1)} \wedge G_{kj} \wedge G_{kl}) + X_{lh} ( C^{(s-1)} \wedge G_{lk} \wedge G_{lj}), $$
and for $h \geq 4$
$$ v^{(2)}_n(e_i \otimes f_h) = Y_i ( C^{(s-2)} \wedge g_{h-3} \wedge G_{kl} \wedge G_{jl} \wedge G_{jk})  + Y_{j} ( C^{(s-1)} \wedge g_{h-3} \wedge G_{kl}) + $$ $$ - Y_{k} ( C^{(s-1)} \wedge g_{h-3} \wedge G_{jl}) + Y_{l} ( C^{(s-1)} \wedge g_{h-3} \wedge G_{jk}). $$ 
For $n=2s+1$ we have, for $h \leq 3$:
$$ v^{(2)}_n(e_i \otimes f_h) = X_{ih} ( C^{(s-1)} \wedge G_{ij} \wedge G_{ik} \wedge G_{il}) + X_{jh} (C^{(s)} \wedge G_{ij}) - X_{kh} (C^{(s)} \wedge G_{ik}) + X_{lh} (C^{(s)} \wedge G_{il}), $$
and for $h \geq 4$:
$$ v^{(2)}_n(e_i \otimes f_h) = Y_{i} (C^{(s)} \wedge g_{h-3}) + Y_j (C^{(s-1)} \wedge g_{h-3} \wedge G_{il} \wedge G_{ik}) + $$ $$- Y_k ( C^{(s-1)} \wedge g_{h-3} \wedge G_{il} \wedge G_{ij}) + Y_l ( C^{(s-1)} \wedge g_{h-3} \wedge G_{ik} \wedge G_{ij}). $$ 
  \end{prop}

Putting together Propositions \ref{hbm3},\ref{hbm1} and \ref{hbm2} and computing the explicit entries of the maps one gets the differentials described in Theorem \ref{thhbm}.


\section{Open subsets of $Spec ({\hat R}_{gen})$}\label{sec:usplit}

The first application of the complexes described in Theorems \ref{thhbn} and \ref{thhbm} is a proof that the open sets $U_{CM}$ and $U_{split}$ of the spectrum of generic ring ${\hat R}_{gen}$ are the same.

Let us consider the spectrum $Spec({\hat R}_{gen})$ of the generic ring constructed in \cite{W18}. We assume we are in Dynkin format for which the complex $\FF^{top}_\bullet$ exists
(i.e. we exclude formats $(1,n,n,1)$ for $n$ odd and the format $(2,5,5,2)$). The space  $Spec({\hat R}_{gen})$ contains two interesting open subsets.
The subset $U_{CM}$ is the set of prime ideals $P$ for which the localization of the dual complex $(\FF^{gen}_\bullet )^*$ at $P$ is acyclic, i.e. the points where the localization of the generic module $H_0 (\FF^{gen}_\bullet)$ is perfect.

There is also another open set $U_{split}$ which consists of points for which the complex $\FF^{top}_\bullet$ is split exact.
Conjecturally these two open sets are equal, see Conjecture $7.4$ in \cite{LW19}.

Here we prove this conjecture for $D_n$ type formats. This result was proved in \cite{CVW18} and it is not difficult to see using the structure of perfect ideals of codimension $3$ with resolutions of these formats and linkage. The point is that our proof is independent of linkage arguments, so it is likely to generalize to other Dynkin formats.

\begin{thm} \label{Ucm-Usplit}
For the formats $(1,n,n,1)$ ($n$ even) and $(1,4,n,n-3)$, the sets $U_{CM}$ and $U_{split}$ are equal.
\end{thm}

\begin{proof} We work with generic ring ${\hat R}_{gen}$ either for the format $(1,n,n,1)$ ($n$ even) or for the format $(1,4,n,n-3)$. Let us denote by $d_1, d_2, d_3$ the differentials in $\FF^{gen}_\bullet$ and by $\partial_1, \partial_2, \partial_3$ the differentials in $\FF^{top}_\bullet$. The ideals $I(d_i)$ (resp. $I(\partial_i)$) denote the ideals of $r_i\times r_i$ minors of these differentials, where $r_i$ is the rank of $d_i$ and $\partial_i$. We denote by $a_2$ the Buchsbaum-Eisenbud multiplier map from ${\hat R}_{gen} \to \bigwedge^{r_2}F^{gen}_1 $ defined in \cite{BE74}. We also denote by $I(a_2)$ the ideal generated by the coordinates of $a_2$.

 
The complement of the set $U_{CM}$ consists of points where the depth of the localization of $I(a_2)$ is at least $3$.
Let us assume that $P_1,\ldots ,P_s$ are all prime ideals of height $2$ containing $I(a_2)$. Their zero sets $Y_1,\ldots ,Y_s$ are all components of codimension $2$ in the complement of $U_{CM}$.
Consider the open set $U(I(\partial_3))$ of $Spec({\hat R}_{gen})$ which is a complement of the zero set $V(I(\partial _3))$ of $I(\partial_3)$. Since $I(\partial_3)$ has depth $3$ in ${\hat R}_{gen}$ the components of $V(I(\partial_3))$ have codimension $3$. Thus the components $Y_1,\ldots ,Y_s$ have to intersect $U(I(\partial_3))$. So we need to see that our equality $U_{CM}=U_{split}$ holds only in points of $U(I(\partial_3 ))$.

Now we use the highest-lowest weight symmetry of ${\hat R}_{gen}$. This ring has an automorphism $\psi :{\hat R}_{gen}\rightarrow {\hat R}_{gen}$ taking differentials of $\FF^{gen}_\bullet$ to differentials of $\FF^{top}_\bullet$.
So the points in $U(I(\partial_3))$ are exactly the points where the complex $\FF^{top}_\bullet$ is a direct sum of a Hilbert-Burch type complex and a split complex. By symmetry $\psi (U(I(\partial_3))$ is $U(I(d_3))$. So we need to analyze the points in $U(I(d_3))$ where the complex $\FF^{top}_\bullet$ resolves Cohen-Macaulay module. But a points in $U(I(d_3))$ are exactly the points where the complex $\FF^{gen}_\bullet$ is of the form given by a direct sum of an Hilbert-Burch complex and a split summand. So we need to see where the complex $\FF^{top}_\bullet$ resolves a Cohen-Macaulay module for the original complex being a direct sum of an Hilbert-Burch complex and an identity map. 
But this is done in previous sections (Theorem \ref{thhbn} for the $(1,n,n,1)$ format and Theorem \ref{thhbm} for the $(1,4,n,n-3)$ format), and it is clear from these results that both complexes $\FF^{top}_\bullet$ are acyclic together with both duals precisely at points where the maximal minors of the Hilbert-Burch matrix become a unit ideal, which correspond to the points in the set $U_{split}$.
\end{proof}

\section{Other applications of the Hilbert-Burch top complex}\label{sec:examples}

In this last section we provide concrete examples of free resolutions of ideals in commutative rings that can be obtain by specializing the complexes described in Theorems \ref{thhbn} and \ref{thhbm}.

For an $n \times n-1$ matrix $M$, denote by $ M_{\hat{k}}^{\hat{i},\hat{j}} $ the minor of $M$ obtained by deleting the rows $i,j$ and the column $k$.

\begin{thm}
\label{appln}
Let $n=2s \geq 4$. Consider a perfect ideal of height two $I=(f_1, \ldots, f_n)$ in a commutative local Noetherian ring $R$ and let $M=(m_{ij})$ be the Hilbert-Burch matrix resolving it. Let $Z$ be an indeterminate over the ring $R$. For every $1 \leq k < r \leq n$, the ideal $$\mathcal{I}= (f_k,f_r) + \sum_{i \neq k,r} (Zf_i) \subseteq R[Z]$$ is resolved by a complex $\FF$ of format $(1,n,n,1).$ 
Moreover, the following are equivalent:
\begin{enumerate}
\item $\FF$ is a minimal free resolution of $\mathcal{I}$.
\item The ideal $(M_{\hat{1}}^{\hat{k},\hat{r}}, \ldots, M_{\hat{n-1}}^{\hat{k},\hat{r}})$ generated by the submaximal minors of $M$ associated to the choice of indices $k,r$, has depth 2.
\end{enumerate}
\end{thm}

\begin{proof}
Consider the complex $\FF^{top}_\bullet$ obtained in Theorem \ref{thhbn}. Specializing the entries of the generic matrix $X$ to the entries of $M=(m_{ij})$, one gets, by the Hilbert-Burch Theorem, that the maximal minor $Y_j$ specializes to the generator $f_j$ of $I$. Write a partition $$\lbrace 1,2,\ldots, n \rbrace = \lbrace k_1,r_1 \rbrace \cup \lbrace k_2, r_2 \rbrace \cup \ldots \cup \lbrace k_s, r_s \rbrace$$ such that $k_1=k$,$r_1=r$ and $\lbrace k_i,r_i \rbrace \cap \lbrace k_j,r_j \rbrace = \emptyset$ for every $i \neq j$.
Specialize the defect variables $b_{ij}$ the following way:
$b_{kr} = Z$, $b_{k_hr_h}= 1$ for every $h=2, \ldots, s$, and $b_{ij}=0$ in all the other cases. Hence, the pfaffians in the variables $b_{ij}$, specialize as $P=Z$ and \[P_{\hat{i},\hat{j}}= \left\{ \begin{array}{cc} Z &\mbox{if }i = k_h,j= r_h \mbox{ for } h = 2,\ldots,s \\
1&\mbox{if } i=k, j=r \\
0&\mbox{otherwise. } \\
\end{array}\right.  \]
It follows that the ideal of $v^{(1)}_{top}$ specializes to $\mathcal{I}$ and it admits a free resolution $\FF$ of length 3 of format $(1,n,n,1)$ given by the specialization of the complex $\FF^{top}_\bullet$. 
We describe, up to changing signs, the rows $R_j$ of the matrix of the second differential of $\FF$. We have: 

$$ R_{k}= \bmatrix Zm_{r1} & \ldots & Zm_{r_{n-1}} & f_k \endbmatrix, R_r = \bmatrix Zm_{k1} & \ldots & Zm_{k,n-1} & f_r \endbmatrix $$ and, for $h=2,\ldots,s$, $$ R_{k_h}= \bmatrix m_{r_h1} & \ldots & m_{r_{h},n-1} & f_{k_h} \endbmatrix, R_{r_h}= \bmatrix m_{k_h1} & \ldots & m_{k_{h},n-1} & f_{r_h} \endbmatrix.$$

The entries of the third differential of $\FF$ are: for $j \leq n-1$, $$ v_j= M_{\hat{j}}^{\hat{k},\hat{r}} + Z \left( \sum_{h=2}^{s} (-1)^{k_h+r_h} M_{\hat{j}}^{\hat{k_h},\hat{r_h}} \right),  v_n=Z. $$

Clearly $\FF$ is a complex and is easy to observe that the second differential of $\FF$ has rank $n-1$. Moreover, all its maximal minors obtained by deleting the last column are contained in the ideal $ZI$, which is contained in only one principal ideal of $R[Z]$, namely $ZR[Z].$ But, if at least one of the minors $M_{\hat{j}}^{\hat{k},\hat{r}} \neq 0$, the maximal minors involving the last column are not contained in $ZR[Z]$, and hence the ideal of maximal minors of the second differential has depth at least two.

Now, $\FF$ is a minimal free resolution of $\mathcal{I}$ if and only if the ideal generated by the entries of the third differential has depth at least 3. Going modulo the last entry $v_n= Z$, we get that this is equivalent to have the ideal 
$(M_{\hat{1}}^{\hat{k},\hat{r}}, \ldots, M_{\hat{n-1}}^{\hat{k},\hat{r}})$ having depth 2.
\end{proof}

\begin{example}
Choosing, in the setting of the above theorem, $M=X$ to be a generic matrix, the depth condition on submaximal minors is satisfied for any choice of indices $k,r$. Setting $Y_1, \ldots, Y_n$ to be the maximal minors of $X$, we get that the ideal $(Y_1,Y_2)+Z(Y_3, \ldots, Y_n)$ has a minimal free resolution of length 3 and format $(1,n,n,1)$.

Another specific example is obtained starting by the ideal of maximal minors of the matrix $$ \bmatrix x & 0 & 0 \\ 0 & y & 0 \\ v & 0 & w \\ a & b & c     \endbmatrix.$$ In this case, the previous theorem implies that the ideal $(Zxyw, Zxyc, y(cv-wa),xbw)$ has a minimal free resolution of format $(1,4,4,1)$.
\end{example}

To give a similar application describing some free resolution of format $(1,4,5,2)$, we consider the complex $\FF^{top}_\bullet$ of Theorem \ref{thhbm} for $m=2$. In this case $\FF^{top}_\bullet$ resolves the ideal $ J= (w_1,w_2,w_3,w_4),  $ where $$ w_i=Y_ic_{12} + Y_jB^{12}_{ik,ir}-Y_kB^{12}_{ij,ir}+Y_rB^{12}_{ij,ik}. $$ 
 For an $4 \times 3$ matrix $M$, denote by $ M_{\hat{k}}^{i,j} $ the $2 \times 2$ minor of $M$ obtained by selecting the rows $i,j$ and removing the column $k$.


\begin{thm}
\label{appl5}
Let $I=(f_1, f_2,f_3, f_4)$ be a perfect ideal of height two in a commutative local Noetherian ring $R$ and let $M=(m_{ij})$ be the Hilbert-Burch matrix resolving it. Let $Z_1,Z_2$ be indeterminates over the ring $R$. Then, the ideal $$\mathcal{J}= (Z_1Z_2f_4,Z_1f_3,Z_2f_2,f_1) \subseteq R[Z_1,Z_2]$$ is resolved by a complex $\FF$ of format $(1,4,5,2).$ 
Moreover, the following are equivalent:
\begin{enumerate}
\item $\FF$ is a minimal free resolution of $\mathcal{I}$.
\item The ideals of minors of $M$, $(M_{\hat{1}}^{2,4},M_{\hat{2}}^{2,4} , M_{\hat{3}}^{2,4})$ and $(M_{\hat{1}}^{3,4},M_{\hat{2}}^{3,4} , M_{\hat{3}}^{3,4})$ are nonzero and at least one of them has depth 2.
\end{enumerate}
\end{thm}

\begin{proof}
Consider the complex $\FF^{top}_\bullet$ obtained in Theorem \ref{thhbm} in the case $m=2$. Specializing the entries of the generic matrix $X$ to the entries of $M=(m_{ij})$, one gets, as in the proof of Theorem \ref{appln}, that the maximal minor $Y_j$ specializes to the generator $f_j$ of $I$.

Consider the following specialization of the defect variables: $b_{12}^1=Z_1$, $b_{34}^1=1$, $b_{13}^2=Z_2$,$b_{24}^2=1$ and $b^u_{ij}=0$ for any other choice of $1 \leq i < j \leq 4$ and $u=1,2$. Then the minors $B^{12}_{ij,kr}$ specialize as $B^{12}_{12,13}= Z_1Z_2$,$ B^{12}_{12,24}= Z_1$, $B^{12}_{34,13}= Z_2 $,$ B^{12}_{34,24}= 1$ and $B^{12}_{ij,kr}=0$ for any other combination of $i,j,k,r$. Hence, the ideal $J$ specializes to the ideal $\mathcal{J}$ and the complex $\FF^{top}_\bullet$ specializes to a complex $\FF$ having differentials: 
$$  d_2= \bmatrix 
m_{41} &  m_{42} & m_{43} & f_3 & -f_2   \\ 
 m_{31}Z_2 &  m_{32}Z_2 & m_{33}Z_2 & -Z_2f_4 & f_1  \\ 
m_{21}Z_1 &  m_{22}Z_1 & m_{23}Z_1 & f_1 & -Z_1f_4  \\ 
m_{11}Z_1Z_2 &  m_{12}Z_1Z_2 & m_{13}Z_1Z_2 & -Z_2f_2 & Z_1f_3 \\ 
 \endbmatrix,
 $$  $$ \footnotesize
 d_3= \bmatrix 
 Z_1M^{12}_1+ M^{34}_1 & -Z_2M^{13}_1- M^{24}_1
 \\  -Z_1M^{12}_2-M^{34}_2 & Z_2M^{13}_2+ M^{24}_2
 \\ Z_1M^{12}_3+ M^{34}_3 & -Z_2M^{13}_3- M^{24}_3
 \\ Z_1   & 
0  
 \\  0  &  -Z_2  
 \endbmatrix.  
   $$
   Calling $R_1,R_2,R_3,R_4$ the rows of $d_2$, it is easy to see that $(f_4Z_1Z_2)R_1+(f_3Z_1)R_2+(f_2Z_2)R_3+(f_1)R_4=0$ and therefore $d_2$ has rank $3$. Moreover, the maximal minors of $d_2$ obtained removing the last two columns are of the form $Z_1Z_2f_i$ while some minors involving the last two columns, if one among $m_{41},m_{42},m_{43}$ is nonzero, are clearly not contained in the ideals generated by $Z_1$ and $Z_2$.  It follows that the ideals of minors of $d_2$ has depth at least $2$.
 Finally, consider the maximal minors of $d_3$ and go modulo the generator $Z_1Z_2$. It is easy to show that ideal generated by the images of the remaining minors in the quotient ring has depth at least 2 if and only if condition 2 of the statement is verified.
\end{proof}

\section*{Acknowledgements}
The authors are supported by the grants MAESTRO NCN - UMO-2019/34/A/ST1/00263 - Research in Commutative Algebra and
Representation Theory and NAWA POWROTY - PPN/PPO/2018/1/00013/U/00001 - Applications of Lie algebras to Commutative Algebra.
The authors would like to thank Ela Celikbas, Lars Christensen, Sara Angela Filippini, Jai Laxmi, Jacinta Torres, and Oana Veliche for interesting conversations about the content of this paper.

\end{document}